\documentclass{amsart}
\usepackage{amsfonts,amsthm,amsmath}
\usepackage{amssymb}
\usepackage{amscd}
\usepackage[utf8]{inputenc}
\usepackage[a4paper]{geometry}
\usepackage{enumerate}
\usepackage[usenames,dvipsnames]{color}
\usepackage{tikz}
\usepackage{soul}
\usepackage{array}
\usepackage{array,tabularx}

\numberwithin{equation}{section}
\numberwithin{equation}{subsection}

\theoremstyle{plain}

\newtheorem{lemma}[equation]{Lemma}
\newtheorem{proposition}[equation]{Proposition}

\newtheorem{corollary}[equation]{Corollary}
\newtheorem{fact}[equation]{Fact}

\newtheorem{thm}[equation]{Theorem}

\newtheorem{lem}[equation]{Lemma}
\newtheorem{prop}[equation]{Proposition}
\newtheorem{clm}[equation]{Claim}

\theoremstyle{definition}
\newtheorem{example}[equation]{Example}
\newtheorem{remark}[equation]{Remark}

\newtheorem{bekezdes}[equation]{}

\def\C{\mathbb C}
\def\Q{\mathbb Q}
\def\R{\mathbb R}
\def\Z{\mathbb Z}

\def\P{\mathbb P}
\newcommand{\calb}{{\mathcal B}}

\newcommand{\calv}{{\mathcal V}}

\newcommand{\calt}{{\mathcal T}}
\newcommand{\calk}{{\mathcal K}}
\newcommand{\cali}{{\mathcal I}}
\newcommand{\calj}{{\mathcal J}}

\newcommand{\calO}{{\mathcal O}}

\newcommand{\calS}{{\mathcal S}}

\newcommand{\caln}{\mathcal{N}}

\newcommand{\bt}{{\mathbf t}}\newcommand{\bx}{{\mathbf x}}

\newcommand{\bZ}{{\mathbb{Z}}}
\newcommand{\bQ}{{\mathbb{Q}}}
\newcommand{\bC}{{\mathbb{C}}}

\author{Tam\'as L\'aszl\'o}
\address{BCAM - Basque Center for Applied Math.,
Mazarredo, 14 E48009 Bilbao, Basque Country – Spain}
\email{tlaszlo@bcamath.org}

\author{J\'anos Nagy}
\address{Central European University, Dept. of Mathematics,  Budapest, Hungary}
\email{nagy\textunderscore janos@phd.ceu.edu}

\author{Andr\'as N\'emethi}
\address{Alfr\'ed R\'enyi Institute of Mathematics,
Hungarian Academy of Sciences,
Re\'altanoda utca 13-15, H-1053, Budapest, Hungary \newline
 \hspace*{4mm} ELTE - University of Budapest, Dept. of Geometry, Budapest, Hungary \newline \hspace*{4mm}
BCAM - Basque Center for Applied Math.,
Mazarredo, 14 E48009 Bilbao, Basque Country – Spain}
\email{nemethi.andras@renyi.mta.hu }

\title{Surgery formulae for the Seiberg--Witten invariant of plumbed 3--manifolds}

\begin{document}

\keywords{normal surface singularities, links of singularities,
plumbing graphs, rational homology spheres, Seiberg--Witten invariant, Poincar\'e series, quasipolynomials, surgery formula, periodic constant}

\subjclass[2010]{Primary. 32S05, 32S25, 32S50, 57M27
Secondary. 14Bxx, 14J80, 57R57}

\begin{abstract}
Assume that $M(\calt)$  is a rational homology sphere
plumbed 3--manifold associated with a connected negative definite graph $\mathcal{T}$.
We consider the combinatorial multivariable Poincar\'e series associated with $\mathcal{T}$ and its counting functions, which encode rich topological information.
Using the `periodic constant'
of the series (with reduced variables) we prove surgery formulae for the normalized Seiberg--Witten invariants: the periodic constant appears as
 the difference of the Seiberg--Witten  invariants associated with
$M(\calt)$ and $M(\calt\setminus \cali)$, where $\cali$ is an arbitrary subset of the set of vertices of $\calt$.

\end{abstract}

\maketitle

\linespread{1.2}


\pagestyle{myheadings} \markboth{{\normalsize T. L\'aszl\'o, J. Nagy, A. N\'emethi}} {{\normalsize Surgery formulae}}


\section{Introduction}

\subsection{}
Surgery formulae for 3--manifolds, focusing on certain numerical or cohomological
invariant are key tools in
low dimensional topology.  They can  serve e.g. in the identification of  invariants, or in the proof of the coincidence of two differently
defined one,
but also in concrete computations of the invariants for certain families of
manifolds.
Some numerical surgery formulae are consequences of cohomological exact sequences,
where the involved  cohomological theories are categorifications of the
corresponding numerical invariants.
E.g., as the Seiberg--Witten invariant admits several categorifications --- the Heegaard--Floer homology
of Ozsv\'ath and Szab\'o, or the monopole homology of Kronheimer and Mrowka, or
(in the case of plumbed manifolds) the lattice cohomology introduced by the
third author ---,  exact sequences in these theories induce surgery formulae
for the Seiberg--Witten invariant as well, see e.g.  \cite{OSz,Greene,Nexseq}.
Usually, such exact triangles compare the invariants of three surgery 3--manifolds, see again \cite{OSz}.

However, for negative definite graph manifolds, one can formulate a
 different type of surgery formula, which is not imposed by purely topological theories
 and it has no extension  (by the knowledge of the authors) to arbitrary 3--manifolds.
It has its roots in complex algebraic/analytic  geometry by surgery formulae associated with
analytic invariants, where certain Hilbert series play crucial role, see e.g. \cite{Ok}.
Using the fact that negative definite graph manifolds are exactly the links of normal surface singularities,
one can try to transport such ideas from the  analytic theory giving rise to purely topological
results. By the new  formula we present,
 the difference of  the Seiberg--Witten invariants of two surgery manifolds
 is determined  from a  multivariable zeta--type
series, which is combinatorially defined from the graph. This series
 is the topological analogue of a Poincar\'e series of
a  multivariable divisorial filtration (of  a local analytic algebra), 
 but in this topological/combinatorial discussion
the analytic part can be totally neglected (however, for such
 connections see
\cite{BN,Nfive,NJEMS,NN1,NOSZ,trieste,NPS}).

In the sequel $M(\calt)$ denotes
a plumbed 3--manifold associated with
a connected negative definite graphs $\calt$. We will assume that $M(\calt)$
is a rational homology sphere (hence $\calt$ is a tree of $S^2$'s).

The series $Z(\bt)$, which guides several topological invariants of $M(\calt)$,
is defined combinatorially from $\calt$, see
 (\ref{eq:1.1}). The number of variables $\{t_v\}_v$ is indexed by the set of vertices $\calv$ of $\calt$. The series
$Z(\bt)$ decomposes
as a sum $Z_h(\bt)$ according to the spin$^c$--structures of $M(\calt)$.

If $S(\bt)=\sum_{l'}s(l')\bt^{l'}$ is a multivariable series, then its counting function $Q(l')$ is defined by $Q(l'_0)=\sum _{l'\not\geq l'_0} s(l')$.
We say that $S$ admits a quasipolynomial (in the cone $\calk$) if for elements $l'_0$ from
a shifted cone of type $ l^*+\calk$ the value $Q(l'_0)$ equals the value
$\mathfrak{Q}(l'_0)$ of a quasipolynomial $\mathfrak{Q}$ (for precise definitions
see \ref{ss:perConst}). In this case we define
the {\it periodic constant } of $S$ (associated with $\calk$) as $\mathfrak{Q}(0)$.
The construction creates a bridge between topological invariants of $M(\calt)$ and generalized
Ehrhart theory of counting functions and their quasipolynomials (see e.g. \cite{LN}).
The above construction applied for $Z_h$  realizes a deep
connection between low dimensional topology and multivariable (Poincar\'e type) series and their   periodic constants. Indeed, by \cite{NJEMS},
the periodic constant of $Z_h$ (associated with the Lipman cone of
$\calt$) equals the normalized Seiberg--Witten invariant
of $M$ (where $h$ indexes the corresponding spin$^c$--structures).
For the precise  statement see Theorems \ref{th:JEMS} and \ref{th:NJEMSThm}. 

The surgery formula which (partly) motivated our research is the following  \cite{BN}.
Let us fix a vertex $v\in\calv$ and consider the graph $\calt\setminus v$ obtained from
$\calt$ by eliminating $v$. Then the difference of the
 normalized Seiberg--Witten invariants of $M(\calt)$ and
$M(\calt\setminus v)$ can be computed as the periodic constant of the one--variable series $Z_h(\bt)|_{t_u=1, \, u\not=v}$ (cf. Theorem \ref{th:BNsurg}).

One of the main results
of the present work
 (see Theorem \ref{th:MainResult})
 is a common generalization of the above results. We fix an arbitrary subset $\cali\subset \calv$ of the vertices of $\calt$, and we prove that the difference of the
 normalized Seiberg--Witten invariants of $M(\calt)$ and
$M(\calt\setminus \cali)$ can be computed as the periodic constant of the
series with reduced variables $Z_h(\bt)|_{t_u=1, \, u\not\in \cali}$.
(Note that theory of quasipolynomials and also the concrete computation of their periodic constants
is much harder in the multivariable case.)
In the case when $\cali$ is the set of nodes of $\calt$ then we recover the `reduction theorem' from \cite{LNRed}.
In the cases when $\calt\setminus \cali$
contains only strings or `rational graphs' (that is, $M(\calt\setminus \cali)$
is an L--space) then the formula simplifies, and we remain with a
closed formula of the normalized Seiberg--Witten invariants of $M(\calt)$ in terms of $Z_h$ with reduced variables.

It is important to mention that the `classical'
exact triangles (like in \cite{OSz}), hence their surgery formulae too,  involve manifolds which are
modified along a knot. This, in the language of plumbing graph  means modification along one  of the
vertices. On the other hand, our formula is more general, since $\cali$ can be an arbitrary subset of vertices.
 Additionally, our formulae separates
the involved spin$^c$--structures (while exact triangles usually mix them).

\subsection{} The organization of the paper is the following. Section
\ref{s:prliminaries} contains preliminaries regarding plumbing graphs, manifolds,
their Seiberg--Witten invariants, and also Poincar\'e series and their periodic constants.
We also recall several key Seiberg--Witten invariant formula which will be used and generalized later.

In section \ref{s:3} we formulate the new results
and we list several applications.
We split the presentation into two steps: we give a `numerical surgery formula'
(Theorem \ref{th:MainResult})
targeting the Seiberg--Witten invariant, and also another surgery formula, which
is a lift of the numerical identity to the level of quasipolynomials
(Theorem \ref{prop:1}).

In sections \ref{s:convMCF}--\ref{s:SurgMCF}--\ref{sec:MCF} we prove the new
 results. In the proof we decompose the counting function into an alternating sum
of `modified counting functions'. In section \ref{s:convMCF} prove a `convexity
property' of these sums, in section  \ref{s:SurgMCF} a surgery formula for them,
and finally in section  \ref{sec:MCF} we finish the proof.

Section \ref{s:ExProof} treats the case when $\cali$ is the set of nodes
(hence $\calt\setminus \cali$ are strings), while
section \ref{s:8} the case when all subgraphs $\calt\setminus \cali$ are rational.
In these cases several vanishing results are established. Here several
 computations are based on the (positive answer to the)
  Seiberg--Witten Invariant Conjecture from \cite{trieste,NCL,BN} and  we also explain  (and use)
  the connections
  with the analytic (singularity theoretical) counterpart as well.

The last section treats the case of numerically Gorenstein graphs, where
some additional nice symmetries and dualities appear.

\subsection*{Acknowledgements}
TL was supported by ERCEA Consolidator Grant 615655 – NMST and
by the Basque Government through the BERC 2014-2017 program and by Spanish
Ministry of Economy and Competitiveness MINECO: BCAM Severo Ochoa
excellence accreditation SEV-2013-0323.

AN  was partially supported by
ERC Adv. Grant LDTBud of A. Stipsicz at R\'enyi Institute of Math., Budapest.
JN was partially supported by NKFIH Grant K119670,
JN and AN were  partially supported by  NKFIH Grant  K112735.

\section{Preliminaries}\label{s:prliminaries}

For more details regarding plumbing graphs, plumbed manifolds and their relations with normal surface singularities see \cite{BN,EN,Nfive,NJEMS,NN1,NOSZ,trieste,NPS,NWsq};
for Poincar\'e series see also \cite{CDGPs,CDGEq}.

\subsection{Plumbing graphs. Plumbed 3--manifolds.}\label{ss:PGP}
We fix a connected plumbing graph $\mathcal {T}$
whose associated intersection matrix is negative definite. We denote the corresponding plumbed 3--manifold by $M=M(\calt)$. In this article we always assume that $M$ is an oriented  rational homology sphere,
equivalently, $\mathcal{T}$ is a tree with all genus
decorations  zero.

We use the notation $\mathcal{V}$ for the set of vertices, $\delta_v$ for the valency of a vertex $v$, and $\mathcal{N}$ for the set of nodes, i.e. vertices with $\delta_v\geq 3$. End--vertices are defined by $\delta_v=1$.

Let $\widetilde{X}$ be the plumbed 4--manifold with boundary associated with
$\mathcal{T}$, hence $\partial \widetilde{X} = M$. Its second
homology $L:=H_2(\widetilde{X},\mathbb{Z})$ is a lattice, freely generated by the classes of 2--spheres $\{E_v\}_{v\in\mathcal{V}}$, with a negative definite intersection form  $(\,,\,)$. Furthermore, $H^2(\widetilde{X},\mathbb{Z})$
can be identified with the dual lattice  $L':={\rm Hom}_\bZ(L,\bZ)=\{l'\in L\otimes\bQ\,:\, (l',L)\in\bZ\}$. It  is generated
by the (anti)dual classes $\{E^*_v\}_{v\in\mathcal{V}}$ defined by $(E^{*}_{v},E_{w})=-\delta_{vw}$, the opposite of the Kronecker symbol.
One has the inclusions  $L\subset L'\subset L\otimes \bQ$, and  $H_1(M,\mathbb{Z})\simeq L'/L$, denoted by $H$. We write $[x]$
for the class of $x\in L'$ in $H$.

For any $h\in H$ let $r_h\in L'$ be its unique representative in
the `semi--open cube'
$\{\sum_vl'_vE_v\in L'\,:\, l'_v\in[0,1)\}$.

$L'$ carries a partial ordering induced by $l'=\sum_vl'_vE_v\geq 0$ if and only if  each $l'_v\geq 0$.

\subsection{The series $Z(\mathbf{t})$.}
The  \emph{multivariable topological Poincar\'e series} is the
Taylor expansion $Z(\mathbf{t})=\sum_{l'} z^\calt(l')\bt^{l'}
\in\bZ[[L']] $ at the  origin of the rational function
\begin{equation}\label{eq:1.1}
f(\mathbf{t})=\prod_{v\in \mathcal{V}} (1-\mathbf{t}^{E^*_v})^{\delta_v-2},
\end{equation}
where
$\bt^{l'}:=\prod_{v\in \mathcal{V}}t_v^{l'_v}$  for any $l'=\sum _{v\in \mathcal{V}}l'_vE_v\in L'$ ($l'_v\in\bQ$).
It decomposes as $Z(\mathbf{t})=\sum_{h\in H}Z_h(\mathbf{t})$, where $Z_h(\mathbf{t})=\sum_{[l']=h}z^\calt (l')\bt^{l'}$. The expression
(\ref{eq:1.1}) shows that  $Z(\mathbf{t})$ is supported in the \emph{Lipman cone} $\mathcal{S}':=\mathbb{Z}_{\geq0}\langle E^{*}_{v}\rangle_{v\in\mathcal{V}}$.
Since  $I$ is negative definite, all the entries of
$E_v^*$ are strict positive, hence $\calS'\subset
\{\sum_vl'_vE_v\,:\, l'_v>0\}\cup\{0\}$.
Thus, for any $x$, $\{l'\in \calS'\,:\, l'\not\geq x\}$ is finite, cf. \cite[(2.1.2)]{NJEMS}.

Fix $h\in H$.  We define a `counting function'
of the coefficients of $Z_h$ by
\begin{equation}\label{eq:countintro}
Q^\calt_h: L'_h:=\{x \in L'\,:\, [x]=h\}\to \bZ, \ \ \ \
Q^\calt_{h}(x)=\sum_{l'\ngeq x,\, [l']=h} z^\calt(l').
\end{equation}

For the motivation of the truncation  $\{l'\ngeq x,\, [l']=h\}$  see
the results below (e.g. \ref{eq:SUM}) or \cite{NCL}.

\subsection{Seiberg--Witten invariants of $M$.}
The 4--manifold $\widetilde{X}$ has a complex structure.
In fact, any such $M(\calt)$ is the link of a complex normal surface singularity
$(X,o)$, which has a resolution $\widetilde{X}\to X$ with
resolution graph $\calt$ (see e.g. \cite{Nfive}). In this analytic case
$\widetilde{X}$ is a smooth complex 2--manifold, and as a real smooth manifold is the disc--plumbed 4-manifold associated with $\calt$.
Let
$K\in L'$ be its canonical cycle. Though the complex structure
(with fixed $\calt$) is not unique,
 $K$ is determined topologically by $L$ via the adjunction formulae $(K+E_v,E_v)+2=0$ for all $v$. Let
$\widetilde{\sigma}_{can}$ be the {\it canonical
$spin^c$--structure on $\widetilde{X}$} identified by $c_1(\widetilde{\sigma}_{can})=-K$,
and let $\sigma_{can}\in \mathrm{Spin}^c(M)$
 be its restriction to $M$, called the {\it canonical
$spin^c$--structure on $M$}. $\mathrm{Spin}^c(M)$ is an $H$--torsor
 with action denoted by $*$.

 We denote by $\mathfrak{sw}_{\sigma}(M)\in \bQ$ the
\emph{Seiberg--Witten invariant} of $M$ indexed by the $spin^c$--structures $\sigma\in {\rm Spin}^c(M)$ (cf. \cite{Lim, Nic04}).
(We will use the sign convention of \cite{BN,NJEMS}.)

In the last years several combinatorial expressions were established for the Seiberg--Witten invariants.   For rational homology spheres,
Nicolaescu \cite{Nic04} showed  that $\mathfrak{sw}(M)$ is
equal to the Reidemeister--Turaev torsion normalized by the Casson--Walker invariant. In the case when $M$ is a negative
definite plumbed rational homology sphere, combinatorial formula for Casson--Walker invariant in terms of the plumbing graph can be found in Lescop
\cite{Lescop}, and  the Reidemeister--Turaev torsion is determined by N\'emethi and Nicolaescu \cite{NN1} using Dedekind--Fourier sums.

A different  combinatorial formula of $\{\mathfrak{sw}_\sigma(M)\}_\sigma$ was proved  in \cite{NJEMS} using qualitative properties of the coefficients of the series $Z(\mathbf{t})$.

\begin{thm} \label{th:JEMS}\ \cite{NJEMS}
For any $l'\in -K+ \textnormal{int}(\mathcal{S}')$
\begin{equation}\label{eq:SUM} - Q_{[l']}(l')=
\frac{(K+2l')^2+|\mathcal{V}|}{8}+\mathfrak{sw}_{[-l']*\sigma_{can}}(M).
\end{equation}
\end{thm}
\noindent If we fix $h\in H$ and we write  $l'=l+r_{h}$ with $l\in L$,
then the right hand side of (\ref{eq:SUM}) is
 a multivariable quadratic polynomial on $L$,
 a fact which
  will be exploited conceptually  next.

\subsection{\bf Periodic constants.}\label{ss:perConst}
A key tool of the present article is an invariant associated with series motivated by properties of Hilbert--Samuel functions used in algebraic geometry and
singularity theory.
 This also creates a bridge with Ehrhart theory and the properties of its qusipolynomials.
It is called the {\it periodic constant} of the series.
For  one--variable series they were introduced
in \cite{NOk,Ok}, see also \cite{BN},
the multivariable generalization is treated in  \cite{LN}.

Let $S(t)=\sum_{l\geq0}c_l t^l \in
\mathbb{Z}[[t]]$ be a formal power series with one variable.
Assume that for some $p\in \mathbb{Z}_{>0}$ the counting function
$Q^{(p)}(n):=\sum_{l=0}^{pn-1}c_l$ is a polynomial $\mathfrak{Q}^{(p)}$
in $n$. Then the constant term $\mathfrak{Q}^{(p)}(0)$ is independent of $p$ and it is called the \emph{periodic
constant} $\mathrm{pc}(S)$ of the series $S$. E.g.,
if $S(t)$ is a finite polynomial, then $\mathrm{pc}(S)$ exists and it equals
$S(1)$.
If the coefficients of $S(t)$ are given by a Hilbert function $l\mapsto c(l)$,
which admits a Hilbert polynomial $H(l)$ with $c(l)=H(l)$ for $l\gg 0$, then
$S^{reg}(t)=\sum_{l\geq 0}H(l)t^l$ has zero periodic constant and
$\mathrm{pc}(S)=\mathrm{pc}(S-S^{reg})+\mathrm{pc}(S^{reg})=(S-S^{reg})(1)$,
measuring the difference between the Hilbert function and Hilbert polynomial.

For the multivariable case we consider a (negative) definite lattice
$L=\Z\langle E_v\rangle_v$,
 its dual lattice $L'$, a series $S(\mathbf{t})\in \bZ[[L']]$
(e.g. $Z(\bt)$), and its well-defined counting function $Q_h=Q_h(S(\bt))$ as in (\ref{eq:countintro}) for fixed $h\in L'/L$.
Assume that there exist
 a real cone $\mathcal{K}\subset L'\otimes\mathbb{R}$ whose affine closure is  top--dimensional,  $l'_* \in \mathcal{K}$,
 a sublattice $\widetilde{L} \subset L$ of finite index,
and  a quasipolynomial $\mathfrak{Q}_h(l)$ ($l\in \widetilde{L}$)
such that $ Q_h(l+r_h)=\mathfrak{Q}_h(l)$  for any
$l+r_h\in (l'_* +\mathcal{K})\cap (\widetilde {L}+r_h)$. Then we say that
 the counting function $Q_h$ (or just $S_h(\mathbf{t})$)
 admits a quasipolynomial in $\mathcal{K}$, namely $\mathfrak{Q}_h(l)$,  and
also  an (equivariant, multivariable)  {\em periodic constant}
associated with $\mathcal{K}$,  which is defined by
\begin{equation}\label{eq:PCDEF}
\mathrm{pc}^{\mathcal{K}}(S_h(\bt)) :=\mathfrak{Q}_h(0).
\end{equation}
 The definition does not depend on the choice of the sublattice $\widetilde L$, 
 which corresponds to the choice of
 $p$ in the one--variable case.
 This is responsible for the name `periodic' in the definition.
 The definition is independent of the choice of $l'_*$ as well.

By general theory of multivariable Ehrhart-type quasipolynomials
(counting special coefficients  of lattice points
 in polytopes attached to $Z(\mathbf{t})$)
 one can construct a conical chamber decomposition of the space $L'\otimes\mathbb{R}$, such that each cone satisfies the above definition
 (hence provides a periodic constant), for details see \cite{LN} or \cite{SzV}.
This decomposition, in principle,  divides $\mathcal{S}'_{\mathbb{R}}
:=\mathcal{S}'\otimes\mathbb{R}$
into several sub--cones (hence, providing different quasipolynomials and periodic constants associated with these sub--cones  of $\mathcal{S}'_{\mathbb{R}}$).
However, Theorem \ref{th:JEMS} guarantees that this is not the case,
the whole $\mathcal{S}'_{\mathbb{R}}$ is a unique chamber (cf. also with \cite{LSz}).
Hence, Theorem \ref{th:JEMS} reads as follows.
\begin{thm}\label{th:NJEMSThm} \ \cite{NJEMS}
The counting function
 of $Z_h(\bt)$ in the cone $S'_{\mathbb{R}}$
admits  the (quasi)polynomial
\begin{equation}\label{eq:SUMQP} \mathfrak{Q}_{h}(l)=
-\frac{(K+2r_h+2l)^2+|\mathcal{V}|}{8}-\mathfrak{sw}_{-h*\sigma_{can}}(M),
\end{equation}
whose periodic constant is
\begin{equation}\label{eq:SUMQP2}
\mathrm{pc}^{S'_{\mathbb{R}}}(Z_h(\mathbf{t}))=\mathfrak{Q}_h(0)=
-\mathfrak{sw}_{-h*\sigma_{can}}(M)-\frac{(K+2r_h)^2+|\mathcal{V}|}{8}.
\end{equation}
\end{thm}
The right hand side of (\ref{eq:SUMQP2}) with opposite sign
is called the  {\it $r_h$--normalized Seiberg--Witten invariant} of $M$.

\subsection{\bf Reduced Poincar\'e series.}\label{ss:rps}
Fix $h\in H$. For any $\cali\subset \calv$, $\cali\not=\emptyset$,
  we define the \emph{reduced rational function}
$$f_h(\mathbf{t}_{\cali}):=f_h(\mathbf{t})|_{t_v=1,v\notin \cali}$$
and its Taylor expansion $Z_h(\mathbf{t}_{\cali})$,
called the \emph{reduced Poincar\'e series}.
Note that $Z_h(\mathbf{t}_{\cali})$ can be obtained as $Z_h(\mathbf{t})|_{t_v=1,v\notin \cali}$ as well
(this is well defined:  the
 summations of the corresponding coefficients are {\it finite}, since
  $Z(\bt)$ is supported on $\calS'$). Also,
it is important to notice that
before the elimination of certain variables, we have to decompose the series $Z(\mathbf{t})=\sum_{h}Z_h(\mathbf{t})$ into its components $Z_h(\bt)$,
since the reduced (total)
series $Z(\bt_{\cali})$ does not contain sufficient information, which might provide the decomposition into its components $\{Z_h(\bt_\cali)\}_h$.

For any such $\cali$, one defines several operators connecting the different lattices. First, we define  the projection (along the $E$--coordinates)  $\pi_{\cali}:\mathbb{R}\langle E_v\rangle_{v\in\mathcal{V}}\to \mathbb{R}\langle E_v\rangle_{v\in\cali}$, denoted also as
 $x\mapsto x|_{\cali}$, by $\sum_{v\in \calv}l_vE_v\mapsto \sum_{v\in\cali}l_vE_v$.
Note that if $\cali$ is identified with the set
of vertices of a subgraph $\calt_\cali$, then $\pi_\cali$ does not preserve
the intersection form in the corresponding lattices $L(\calt)$ and $L(\calt_\cali)$.

 $\pi_\cali$  provides the
`projected (real) Lipman cone'
$\pi_{\cali}(S'_{\mathbb{R}})$.

We wish to understand what happens with the information coded in $Z_h$ after
elimination certain variables. The next results, as a prototype, shows that
under certain reduction the `Seiberg--Witten information' survives:
if the set of nodes $\caln$ 
is non--empty then for  $\cali=\caln$ one has the following.
\begin{thm}\label{th:RedThm} \ \cite{LN}
The counting function
 of $Z_h(\bt_{\caln})$ in the cone $\pi_{\caln}(S'_{\mathbb{R}})$
admits  a quasipolynomial and a periodic constant,   and
$$\mathrm{pc}^{\pi_{\mathcal{N}}(S'_{\mathbb{R}})}(Z_h(\mathbf{t}_{\mathcal{N}}))=\mathrm{pc}^{S'_{\mathbb{R}}}(Z_h(\mathbf{t}))=
-\mathfrak{sw}_{-h*\sigma_{can}}(M)-\frac{(K+2r_h)^2+|\mathcal{V}|}{8}.$$
\end{thm}
This result has the following  advantages:
the number of reduced variables (i.e. number of nodes) usually is
considerably less than the number of vertices, a fact which reduces the
complexity of the calculations.
Moreover, the reduced series reflects more conceptually
the complexity of the manifold $M$ (using  only one variable for each Seifert 3-manifold piece in its JSJ--decomposition).  Furthermore,
the reduced series can be compared/linked with other (geometrically
or analytically defined)
 objects as well (see e.g. \cite{BN,NPS}).

We generalize this result in two directions: first, we replace $\caln$ by an arbitrary subset $\cali\not=\emptyset$, and second, we lift the identity from the numerical
periodic constant level to the quasipolynomial level.

 \subsection{\bf A surgery formula associated with the elimination of a vertex.}
 \label{ss:surgfor}
 Surgery formulae for a certain 3--manifold invariant, in general,
 compare the invariant of $M$  with the invariants of different surgery modifications of $M$.
 In the case of plumbed 3--manifolds, one compares the invariants associated with 3--manifolds obtained by different
 modifications of the graph. The `standard' topological surgery formulae for the Seiberg--Witten invariant (induced by exact triangles of certain cohomology theories,  cf. \cite{OSz,Greene,Nexseq})
 compare the invariants of three such 3--manifolds.
 Furthermore, in these approaches, one cannot separate
a  certain fixed $spin^c$ structure, the theory mixes always several of them.
(See also \cite{Turaev}.)
  The next formula (and our
 generalizations as well) are different: they
 compare the Seiberg--Witten invariant of two 3--manifolds via
 an `algebraic' term defined as the  periodic constant of a series
 (and they split according to the $spin^c$--structures).

Let us fix  $\cali\subset \calv$. The set of vertices $\calv\setminus \cali$
determines the  connected full subgraphs $\{\calt_i\}_i$, $\cup_i\calt_i=\calt\setminus \cali$.
For each $i$ we consider the
inclusion operator $j_i:L(\calt_i)\to L(\calt)$,
$E_v(\calt_i)\mapsto E_v(\calt)$, identifying naturally
the corresponding $E$--base elements in the two graphs. This
preserves the intersection forms. Let $j_{i}^*:
L'(\calt)\to L'(\calt_i)$ be its dual, defined by
$j_{i}^*(E^*_{v}(\calt))=E^*_{v}(\calt_i)$ if $v\in\calv(\calt_i)$, and
$j_{i}^*(E^*_{v}(\calt))=0$ otherwise.
Then $(j^*_{i}(l'), l)_{\calt_i}=(l',j_{i}(l))_{\calt}$
for any $l'\in L'(\calt)$ and $l\in L(\calt_i)$.

Let us start with an arbitrary $spin^c$--structure
$\widetilde{\sigma}$ on $\widetilde{X}$. Since ${\rm Spin}^c(\widetilde{X})$
is an $L'$--torsor, there is a unique $l'\in L'$ such that
$\widetilde{\sigma}=l'*\widetilde{\sigma}_{can}$. Its restriction to
${\rm Spin}^c(M)$ is $\sigma=[l']*\sigma_{can}$. We also refer to $\widetilde{\sigma}$ as the extension of $\sigma$. 
Since $\widetilde {X}(\calt_i)$ can be regarded as a small tubular neighbourhood
of those $E_v$ which are contained in $\calv(\calt_i)$, $\widetilde{\sigma}$ has restrictions $\widetilde{\sigma}_i$ to each $\widetilde{X}(\calt_i)$ too.
Since the canonical $spin^c$--structure of $\widetilde X$
 restricts to the canonical $spin^c$--structure $\widetilde\sigma_{can,i}$
of $\widetilde{X}(\calt_i)$, $\widetilde\sigma=l'*\widetilde{\sigma}_{can}$
 restricts to $\widetilde\sigma_i:=j_{i}^*(l')*\widetilde\sigma_{can,i}
 \in {\rm Spin}^c(\widetilde{X}(\calt_i))$,
 whose restriction to the boundary $M_i=M(\calt_i)=\partial \widetilde{X}_i$ is
$\sigma_i=[j_{i}^*(l')]*\sigma_{can,i}$.

 Having these general definitions,  let us consider first the particular case of
 $\cali=\{v\}$ ($v\in\calv$), and,  as above,
  let $\{\calt_i\}_i$ be the connected components of $\calt\setminus \{v\}$.
 The following surgery formula was one of the motivations of our main result.
 Below the reduced $\bt_\cali$ has  only one variable, namely $t_v$,
 and the corresponding periodic constant is computed by the `easy'
 definition of the one--variable series.

 \begin{thm}\label{th:BNsurg} \ \cite{BN} Fix any $h\in H$
 and extend $h*\sigma_{can}\in {\rm Spin}^c(M)$ as $\widetilde{\sigma}:=
 r_h*\widetilde{\sigma}_{can}\in {\rm Spin}^c(\widetilde{X})$.
 Consider also $\cali=\{v\}\subset  \calv$ and the corresponding restrictions of $\widetilde{\sigma}$ to $\cup_i M(\calt_i)$. Then
 the series $Z_h(t_\cali)=Z_h(t_v) $ admits a periodic constant, and
 \begin{equation*}
\begin{split}
\mathfrak{sw}_{-h*\sigma_{can}}(M)+\frac{(K+2r_h)^2+|\mathcal{V}|}{8}
 =\ &
 \sum_{i}\,  \Big( \mathfrak{sw}_{-[j^*_{i}(r_h)]*\sigma_{can,i}}(M_i)+\frac{(K(\mathcal{T}_i)
  + 2j^*_i(r_h))^2+
 |\mathcal{V}(\calt_i))|}{8}\Big)\\
 &- \textnormal{pc}(Z_{h}(t_{v})).
 \end{split}
\end{equation*}
\end{thm}
\noindent
(Note that usually $j^*_i(r_h)\not= r_{[j^*_i(r_h)]}$, see below.)
This will be generalized to arbitrary $\cali\not=\emptyset $ and to an arbitrary extension
$\widetilde{\sigma}:= l'*\widetilde{\sigma}_{can}$  ($l'\in L'$) of $h*\sigma_{can}$.

 \section{The main result: the new surgery formulae}\label{s:3}

 \subsection{}
First we  state a consequence of our Main Theorem \ref{prop:1},
which is still sufficient general to generalize all the previous results.

 We will use the notations of the previous section. Let
 $\cali\subset \calv$ be an arbitrary non--empty subset and write
 $\calt\setminus \cali$ as the union of full
  connected subgraphs  $\cup_i\calt_i$. Moreover, we fix $h\in H$ as well.
 \begin{thm}\label{th:MainResult}
 The series $Z_{h}(\mathbf{t}_\cali)$ admits a periodic constant in the
 real cone $\pi_\cali(\calS'_{\R})$, and
\begin{equation*}
\begin{split}
\mathfrak{sw}_{-h*\sigma_{can}}(M)+\frac{(K+2r_h)^2+|\mathcal{V}|}{8}
 =\ &
 \sum_{i} \Big( \mathfrak{sw}_{-[j^*_{i}(r_h)]*\sigma_{can,i}}(M_i)+\frac{(K(\mathcal{T}_i) + 2j^*_i(r_h))^2+
 |\mathcal{V}(\calt_i)|}{8}\Big)\\
 &- \textnormal{pc}^{\pi_\cali(\calS'_{\R})}(Z_{h}(\mathbf{t}_\cali)).
 \end{split}
\end{equation*}
 \end{thm}
\begin{example}\label{ex:partcases} Let us consider the following particular cases.

(1) \ Assume that $\cali=\calv$. Then each  $\calt_i$ is empty, and we recover Theorem \ref{th:NJEMSThm} proved in  \cite{NJEMS}.

(2) \ Assume that $\cali$ consists of one vertex. Then we recover Theorem
\ref{th:BNsurg} proved in  \cite{BN}.

(3) \ Assume that $\cali=\caln$. Then we recover Theorem \ref{th:RedThm},
proved in \cite{LN}, once we verify for each $i$ the vanishing
\begin{equation}\label{eq:vanishing}
 \mathfrak{sw}_{-[j^*_{i}(r_h)]*\sigma_{can,i}}(M_i)+\frac{(K(\mathcal{T}_i) + 2j^*_i(r_h))^2+
 |\mathcal{V}(\calt_i)|}{8}=0.
\end{equation}
This vanishing is well--known for $h=0$ (see e.g.  \ref{ss:8.2} or
Remark \ref{rem:lc}), but it is not
evident at all (at least for the authors) for arbitrary $h$.
It will be proved in Section \ref{s:ExProof} by analytic methods.

In this case the graph $\calt_i$ is a string, hence $M(\calt_i)$ is a lens space.
The difficulty in the vanishing (\ref{eq:vanishing}) is that $r_h$ is a global object induced from $\calt$, and for any fixed $\calt_i$ is not clear at all what classes of $L'(\calt_i)$ might appear as $j^*_i(r_h)$ for certain `extension graphs' $\calt$. (If $h=0$ then $j^*_i(r_h)=0$ as well, which simplifies the
situation: (\ref{eq:vanishing}) is a statement regarding merely a lens space, and
 it  follows e.g. from \eqref{eq:SWICrat}.)

(4) For generalization of (3) for the case when
 all subgraphs $\calt_i$ are {\it rational}, see Section \ref{s:8}.
\end{example}

\subsection{} The above formula from Theorem
\ref{th:MainResult}, which  targets numerical invariants, is a  consequence of a  general `{\it surgery identity of quasipolynomials}'.
This is the subject of the Main Theorem \ref{prop:1}.

Similarly to the counting functions defined in (\ref{eq:countintro}), we
 set for any $h$ and $\cali\subset\calv$, $\cali\not=\emptyset$,
 \begin{equation}\label{eq:countintro2}
Q^\calt_{h,\cali}: L'_{h}\to \bZ, \ \ \ \ Q^\calt_{h,\cali}(x):=\sum_{l'|_\cali\ngeq x |_\cali,
\, [l']=[x]} z^\calt(l').
\end{equation}
Note that $Q^\calt_{h,\cali}$ depends only on the reduced series
$Z_h(\bt_\cali)$: it is its counting function.

The setup of the next statement is the following.
We fix $h\in H$, and we choose
 $l'_0=\sum_{v\in\calv} a_vE^*_v\in L'$ with
 $[l'_0]=h$. We also fix
$\cali\subset \calv$, $\cali\not=\emptyset$,  and $\calt\setminus \cali=\cup_i\calt_i$.

\begin{thm}\label{prop:1} For any  $l'_0$ with all $a_v$ sufficiently large
 one has the identity
\begin{equation}\label{eq:proof1}
Q^\calt_{[l'_0]}\,(l'_0)=Q^\calt_{[l'_0],\cali}\,(l'_0)+\sum_i \
Q^{\calt_i}_{[j^*_i(l'_0)]}(j^*_i(l'_0)).
\end{equation}
\end{thm}
\subsection{} \label{ss:discussion}
Let us deduce some consequences and corollaries.

Write $l'_0=r_h+l$.
By Theorem \ref{th:JEMS} (see Theorem \ref{th:NJEMSThm} too),
 for all $a_v$ large,   $Q^\calt_{[l'_0]}(l'_0)$ equals
the (quasi)polynomial $\mathfrak{Q}^\calt_h(l)$, and the same is true for
each term in the last sum (by the same theorem applied for $\calt_i$).
Therefore, the identity (\ref{eq:proof1}) guarantees that
 for all $a_v$ large $Q^\calt_{[l'_0],\cali}(l'_0)$ is a quasipolynomial as well.

 \begin{corollary}\label{cor:QP} For $h\in H$ fixed and   $l'_0=r_h+l$, $l\in L$,
 if all  $a_v$ are sufficiently large,
 $Q^\calt_{h,\cali}(l'_0)$ equals
 a quasipolynomial  $\mathfrak{Q}^\calt_{h,\cali}(l)$
 defined on $L$, where
 \begin{equation}\label{eq:QPrestr}
\begin{split}
- \mathfrak{Q}^\calt_{h,\cali}(l):=\ &
\mathfrak{sw}_{-h*\sigma_{can}}(M)+\frac{(K+2r_h+2l)^2+|\mathcal{V}|}{8}\\
-\ & \sum_{i} \Big( \mathfrak{sw}_{-[j^*_i(r_h+l)]*\sigma_{can,i}}(M_i)+\frac{(K(\mathcal{T}_i) +
 2j^*_i(r_h+l))^2+
 |\mathcal{V}(\calt_i)|}{8}\Big).
 \end{split}
 \end{equation}
 \end{corollary}
Except for the `constant term', $\mathfrak{Q}^\calt_{h,\cali}(l)$
is a multivariable quadratic polynomial. However, the constant term is a
`periodic' function $L\to \Q$.  Indeed, $L\mapsto \sqcap_i \,{\rm Spin}^c(M_i)$,
$l\mapsto \{[j^*_i(r_h+l)]*\sigma_{can,i}\}_i$, in general,
 is not constant. However, if we define
$\widetilde{L}\subset L$ as the kernel of the composition
\begin{equation}\label{eq:kerner}
L\longrightarrow \oplus_i\, L'(\calt_i)
\longrightarrow
\oplus_i\, L'(\calt_i)/L(\calt_i),
\ \ \ l\mapsto \oplus_i [j^*_i(l)],
\end{equation}
then the  `constant term' of
$\mathfrak{Q}^\calt_{h,\cali}(l)$ restricted to any class of type
$l_0+\widetilde{L}$ of $L$ is constant.
\begin{corollary}\label{cor:QP2} For $h\in H$ fixed the counting function
 $Q^\calt_{h,\cali}(l'_0)$ ($l'_0\in L'_h$) of $Z_h(\bt_\cali)$
 admits a quasipolynomial in the Lipman cone $\calS'_{\R}$, namely
 $\mathfrak{Q}^\calt_{h,\cali}(l)$, and its periodic constant satisfies
 \begin{equation*}
\begin{split}
-{\rm pc}^{\calS'_{\R}}(Z_h(\bt_{\cali}))=
- \mathfrak{Q}^\calt_{h,\cali}(0)=\ &
\mathfrak{sw}_{-h*\sigma_{can}}(M)+\frac{(K+2r_h)^2+|\mathcal{V}|}{8}\\
-\ & \sum_{i} \Big( \mathfrak{sw}_{-[j^*_i(r_h)]*\sigma_{can,i}}(M_i)+\frac{(K(\mathcal{T}_i) +
 2j^*_i(r_h))^2+
 |\mathcal{V}(\calt_i)|}{8}\Big).
 \end{split}
 \end{equation*}
 \end{corollary}

Note that this is not the statement of Theorem \ref{th:MainResult} yet.
In order to conclude Theorem \ref{th:MainResult} we make the following discussion.
Above we considered
 $Q^\calt_{[l'_0],\cali}(l'_0)$, and its quasipolynomial, as functions in $l'_0\in L'_h$, or in $l=l'_0-r_h\in L$.
 This is the right point of view: when we take periodic constants of a sum of different quasipolynomials, one has to consider this operation
 in the same lattice. In this way the periodic constant will behave as an additive operator,
 cf. Remark  \ref{rem:harom}.

 However,
 note that $Q^\calt_{[l'_0],\cali}\,(l'_0)$ basically counts the
 coefficients of the reduced series $Z_h(\bt_\cali)$, hence it can be considered also as a counting function defined on the lattice $ L|_\cali$, via
$ l'_0|_{\cali}=l|_\cali+r_h|_{\cali}$.
 In this way, its periodic constant in this lattice should be computed by
 substituting into $l'_0|_\cali$ its representative in the semi-open cube
 associated with variables $\cali$.
 But, the point is that
  this is exactly $r_h|_{\cali}$ (since all the entries
 of $r_h|_\cali$ are automatically in $[0,1)$). Hence, the two periodic constant
 (computed in $L$ or $L|_\cali$) agree and provide
 \begin{equation}\label{eq:PcPc}
 {\rm pc}^{ \calS'_\R}(Z_h(\bt_\cali))=
 {\rm pc}^{\pi_\cali( \calS'_\R)}(Z_h(\bt_\cali)).\end{equation}
This fact, together with Corollary \ref{cor:QP2} prove Theorem \ref{th:MainResult}.
\begin{remark}\label{rem:harom}
 (1) The analogue of (\ref{eq:PcPc}) for the subgraphs
$\calt_i$ (that is, for the operator $j_i^*$ instead of $\pi_\cali$) is not valid.
Let us assume e.g. that  $l'_0\in r_h+\widetilde{L}$, hence $[j^*_i(l'_0)]=
[j^*_i(r_h)]$  is constant, say  $h_i\in H(\calt_i)$.
Consider the following expression
 valid for  any $\tilde{l}'_{0}\in r_{h_i}+L(\calt_i)$ with large coefficients:
\begin{equation}\label{eq:UJ}
Q^{\calt_i}_{[\tilde{l}'_{0}]}(\tilde{l}'_{0})=
-\big(K(\calt_i)+2 \tilde{l}'_{0})^2+|\calv(\calt_i)|\big)/8-
\mathfrak{sw}_{-h_i*\sigma_{can,i}}(M_i).\end{equation}
It can be considered in two different lattices.
First, the right hand side is the quasipolynomial $\mathfrak{Q}^{\calt_i}_{h_i}(\tilde{l}'_{0}-r_{h_i})$
associated with the lattice $L(\calt_i)$.
On the other hand,
if we substitute into  $\tilde{l}'_{0}\in
L'(\calt_i)$ the restriction $j^*_i(l'_0)$, it appears as a quasipolynomial
in variable $l'_0-r_h\in L(\calt)$
(this expression  appears in (\ref{eq:proof1})).
 In the first case, in  $L(\calt_i)$,  its periodic constant is $\mathfrak{Q}^{\calt_i}_{h_i}(r_{h_i})$, while in the second case, in the lattice $L$, it is  $\mathfrak{Q}^{\calt_i}_{h_i}(j^*_i(r_h))$.
 Note that usually $r_{h_i}\not=j_i^*(r_h)$, cf. Example \ref{ex:graph2}.

 The message  is the following: when we take periodic constants of a sum of different quasipolynomials, one has to consider this operation
 in the same lattice.
 However, if one of the periodic constants is needed to
 be reinterpreted as a periodic constant in a different lattice then one has to be aware of the fact that the ${\rm pc}$--operation commutes  with projections of type $\pi_\cali$, but usually not with operators of type $j_i^*$.

(2) One has the following identity (for $\calt$, and similar expressions for any $\calt_i$)
\begin{equation*}\frac{(K+2l')^2+|\mathcal{V}|}{8}=
\frac{K^2+|\mathcal{V}|}{8}-\chi(l'),
\end{equation*}
where $\chi(l'):=-(l',l'+K)/2$ is the `Riemann-Roch expression'
 for any $l'\in L'$.

(3) For certain surgery formulae regarding the invariant $K^2+|\calv|$ see e.g.
\cite[\S 5]{BN}.

\end{remark}
\subsection{}\label{ss:anylift} By different choices of $h\in H$ and of liftings
 $l'_0=r_h+l\in L'$, the possible $spin^c$--structures $\widetilde{\sigma}=l_0'*\widetilde{\sigma}_{can}$ fill in ${\rm Spin}^c(\widetilde{X})$ completely.
$\widetilde{\sigma}$  extends
$h*\sigma_{can}$, and its  restriction to
${\rm Spin}^c(M_i)$ are  $[j^*_i(l'_0)]*\sigma_{can,i}$.
Hence the quasipolynomial identity
(\ref{eq:QPrestr}), for any fixed $\widetilde{\sigma}$,
 can be regarded as a surgery formula of the
Seiberg-Witten invariants connecting $(M, h*\sigma_{can})$ and
$\{(M_i,[j^*_i(l'_0)] *\sigma_{can,i})\}_i$ with correction term
$-\mathfrak{Q}^\calt_{h,\cali}(l)$, computable from the quasipolynomial of
$Z^\calt_h(\bt_\cali)$.

\subsection{} The expression $\mathfrak{Q}^\calt_{h,\cali}(l)$ in
Corollary \ref{cor:QP}
can be rewritten in terms  of certain
periodic constant computable from $\bt_\cali^{-l|_{\cali}}\cdot Z_h^\calt(\bt_\cali)$ as follows.

Assume that $S(\bt)=\sum_{l\in {\mathcal K}}c^S(l)\bt^l$ is a series in
variables $l\in L=\Z\langle E_v\rangle_v$ supported on the cone ${\mathcal K}\subset L\otimes \R$.
We assume that $\mathcal{K}=\R_{\geq 0}\langle V_j\rangle_j$, where all the entries of each $V_j$ are positive.
Let $Q^S(l)=\sum_{\tilde{l}\not\geq l}c^S(\tilde{l})$ be its counting function, and assume that it admits the quasipolynomial $\mathfrak{Q}^S(l)$, which satisfies
 $\mathfrak{Q}^S(l)=Q^S(l)$ in a shifted cone of type $l_*+{\mathcal K}$.
 Then, in a convenient shifted cone,  for any fixed $l_0\in L$ one has
$$\mathfrak{Q}^S(l+l_0)=\sum_{\tilde{l}\not\geq l+l_0}c^S(\tilde{l})=
\sum_{\tilde{l}\not\geq l}c^{\bt^{-l_0}S}(\tilde{l}).$$
Usually, $\bt^{-l_0}S(\bt)$ is not a series (it is a Laurent series),
let $\bt^{-l_0}S(\bt)|_{\geq 0}$ and $\bt^{-l_0}S(\bt)|_{\not\geq 0}$
be its decomposition according to its support.  Then  $\bt^{-l_0}S(\bt)|_{\geq 0}$ is a series, while $\bt^{-l_0}S(\bt)|_{\not\geq 0}$ is a finite Laurent polynomial. (E.g., if $l_0\leq 0$ then $\bt^{-l_0}S|_{\not\geq 0}$ is identically zero,
however, in general it is not.) Furthermore,
for $l$ with large  coefficients, $\sum_{\tilde{l}\not\geq l}c^{\bt^{-l_0}S|_{\not\geq 0}}(\tilde{l})=(\bt^{-l_0}S|_{\not\geq 0})({\bf 1})$
(i.e. one substitutes for each $t_v=1$). This proves the following fact.
\begin{proposition}\label{prop:PCcut}
Under the above notations, for any $l_0\in L$ the series
$\bt^{-l_0}S(\bt)|_{\geq 0}$ admits a quasipolinomial
and a periodic constant in the cone $\mathcal{K}$ and
$$\mathfrak{Q}^S(l_0)=(\bt^{-l_0}S|_{\not\geq 0})({\bf 1})+{\rm pc}^{\mathcal {K}}
(\bt^{-l_0}S|_{\geq 0}).$$
\end{proposition}

Using this identity,  Corollary \ref{cor:QP} (and \ref{ss:anylift} as well)
can be  modified  accordingly.

\subsection{Modified counting functions.} \label{ss:4.4}

 We say that $a,b\in \R^k$ satisfies $a\prec b$ if for {\it all}
 coordinates we have $a_v<b_v$. By inclusion--exclusion principle, a
 sum of type  $Q^\calt_{[l'_0],\cali}\,(l'_0)$
 can be rewritten
  as
 $$
 \sum_{l'|_{\cali}\not\geq l'_0|_{\cali}} z^\calt(l')
=\sum_{\exists w\in \cali\,:\, l'|_{w}< (l'_{0})|_w} z^\calt(l')=
 \sum_{\emptyset \not=\calj\subset \cali} \ (-1)^{|\calj|+1}
 \sum _{l'|_\calj \prec l'_0|_\calj} \ z^\calt(l'),$$
 where everywhere in the summations $[l']=[l_0']$.
This motivates to define (for $l'_0\in L'$ with $[l'_0]=h$)
 the `modified counting functions'
$$q^\calt_{h,\calj} \, (l'_0):=\sum_{l'|_\calj \prec\,   l'_{0}|_\calj, \ [l']=[l'_0]} \ z^\calt(l').$$
There are similar expressions for the terms $Q^{\calt_i}_{[j^*_i(l'_0)]}(j^*_i(l'_0))$  of (\ref{eq:proof1}) as well.
Hence,  we can rewrite the wished identities in terms of modified counting
functions. The point is that we will prove the corresponding identities
 for these modified counting functions.
 Their advantage is that they satisfy certain  `convexity'
properties, which generate a lot of cancellations. They
will be treated in  Section \ref{s:convMCF}, a part which
constitutes also the start of the proof of (\ref{eq:proof1}).

 \begin{remark}
 In \cite{LSz} the expression $q^\calt_{h,\cali} \, (l'_0)$ is called the `coefficient function', since it is the coefficient of $\mathbf{t}^{l'_0}_{\cali}$ in the Taylor expansion of $f_h(\mathbf{t}_{\cali})\cdot\prod_{v\in\cali}\mathbf{t}^{E_v}_{\cali}/(1-\mathbf{t}^{E_v}_{\cali})$.
\end{remark}

\section{A `convexity' property of the modified
counting functions}\label{s:convMCF}

\subsection{Some terminology. Multiplicity systems.}\label{ss:terminology}  Assume that $\calt$ is a graph as above, and $l=\sum_vm_vE_v\in L$ is an integral cycle,
which in the dual base is
$l=\sum_vc_vE^*_v$. For each $v\in\calv$ in $\widetilde{X}$ we consider
 2--discs (cuts) $\{C_{v,i}\}_{i=1}^{k_v}$, each of them
 intersecting $E_v$  transversally  in  generic points and with $\partial C_{v,i}\subset \partial \widetilde{X}$. Then, whenever
   $\sum_i c_{v,i}=c_v$ for all $v$,
   $C(l) :=\sum_v (m_vE_v+\sum_ic_{v,i}C_{v,i})$
 is a relative cycle in $\widetilde{X}$ with $(C(l),E_v)=0$ for all $v$ (hence its class in $H_2(\widetilde{X}, \partial \widetilde{X},\Z)$ is zero).
 Each $\partial C_{v,i}\subset \partial \widetilde{X}$ is a link component in
$ \partial \widetilde{X}$, and their collection endowed with the
multiplicities  $\{c_{v,i}\}_{v,i}$ forms a multilink
with `{\it multiplicity system}'  $\{m_v,c_{v,i}\}_{v,i}$
\cite{EN}, see also \cite{P} for the non--integral  homology sphere case.

If $c_{v,i}\geq 0$ and at least one inequality is strict,  then each $m_v>0$ too
(use the fact that the entries of $E^*_v$ are positive).
 Moreover, under the same hypothesis, the multilink is fibered.

Fix  $C(l) $ and a multiplicity system as above. Let $\calt'$ be a
full connected subgraph of $\calt$ and $\widetilde{X}(\calt')$  a small tubular neighbourhood  of $\cup_{v\in \calv(\calt')}E_v$ in $\widetilde{X}$. Then $C(l)$ induces a homologically trivial relative cycle
 $C(l)\cap \widetilde{X}(\calt')$ (by multiplicity preserving intersection)
 in $\widetilde{X}(\calt')$,
hence a multiplicity system  of $\calt'$. In a different language, the
$E_v$, respectively the $E_v^*$--multiplicities of $l$ after restriction are
the following.
A cut $C_v$ is preserved with its multiplicity $c_v$
if $v\in \calv(\calt')$, otherwise it
becomes empty.
The restriction of $E_v$ becomes empty if $E_v\cap E(\calt')=\emptyset$, it becomes a cut
of $E_w$ with multiplicity $m_v$ if $E_v\cap E(\calt')$ is the point $E_v\cap
 E_w$, and it remains $E_v$ with its multiplicity $m_v$ if $v\in\calv(\calt')$.
(Homologically, this is the operator $j^*$ associated with the inclusion
$j:L(\calt')\hookrightarrow L(\calt)$.)

%

\subsection{}\label{ss:5.1}
  Fix  $\calt $ as in  \ref{ss:PGP}.
  Let $\calt_2$ be a connected full subgraph
  of $\calt$ with vertices $\calv_2$ and `boundary'
  $$\calb:=\{u\in\calv_2\,:\  \exists \ w\not\in \calv_2 \ \mbox{adjacent to $u$ in $\calt$}\}.
  $$
  For any fixed $u\in\calb$, $\calt_{1,u}$ denotes that full connected subgraph
  of $\calt$, which contains $u$ and all the connected components of $\calt\setminus \calt_2$, which have adjacent vertices with $u$.
  Write $\calv_{1,u}=\calv(\calt_{1,u})$.

As above, $Z(\bt_{\calv_2})$ is the
reduction of the series of $\calt$ to the variables indexed by $\calv_2$.

For simplicity, we use the same notation $l'|_u$ for the $E_u$-coefficient of $l'\in L'$ too (cf. \ref{ss:rps}).

\begin{lemma}\label{lem:zetadeg}
(a) Any element from the support of $Z(\bt_{\calv_2})$ can be written in a unique way as $\sum_{v\in \calv_2} r_vE^*_v|_{\calv_2} $ for certain coefficients
$r_v\in\Q_{\geq 0}$.

(b) Fix $u\in\calb$. Let
  $\delta_{2,u}$ be the number of edges adjacent to $u$ but sitting in $\calt_2$.
 Assume that  $\delta_{2,u}\geq 2$.
If \ $\sum_{v\in \calv_2} r_vE^*_v|_{\calv_2} $ is in the support of $Z(\bt_{\calv_2})$, then $r_u\cdot
E^*_u|_u\leq \sum_{v\in\calv_{1,u}} (\delta_v-2)E^*_v|_u$.
\end{lemma}
\begin{proof} (a)
For $u\in\calb$ and $v\in \calv_{1,u}$,
the cycle $E^*_v\cdot E^*_u|_u- E^*_u\cdot E^*_v|_u\in L\otimes \Q$ is supported in $\calv_{1,u}\setminus u$.
(Indeed, the $E_u$--multiplicity of $l:=E^*_v\cdot E^*_u|_u- E^*_u\cdot E^*_v|_u$
is zero. Therefore, the restriction of the multiplicity system $C(l)$ to $\calv_{1,u}\setminus u$ has no cuts, has it is identically zero, cf. \ref{ss:terminology}.)
In particular,
\begin{equation}\label{eq:staridentity}
E^*_v|_{\calv_2}= E^*_u|_{\calv_2}\cdot (E^*_v|_u/ E^*_u|_u).\end{equation}
Next, write $Z(\bt)$ as $Z_2(\bt)\cdot \prod_{u\in \calb}Z_{1,u}(\bt)$, where
$Z_{1,u}(\bt):= \prod_{v\in\calv_{1,u}} (1-\bt^{E^*_v})^{\delta_v-2}$ and
$Z_2(\bt):= \prod_{v\in\calv_2\setminus \calb} (1-\bt^{E^*_v})^{\delta_v-2}$.
 Hence, in the support of $Z(\bt_{\calv_2})$, $Z_2(\bt_{\calv_2})$
 contributes with  $\{E_v^*|_{\calv_2}\}_{v\in \calv_2\setminus \calb}$,
while,
 $Z_{1,u}(\bt_{\calv_2})$ with $E_u^*|_{\calv_2}$ for each $u\in\calb$.
Moreover,
$\{E^*_v|_{\calv_2}\}_{v\in \calv_2}$ are linearly independent.
 Indeed, if  $\sum_{v\in \calv_2}r_vE^*_v|_{\calv_2}=0$ then
$x:=\sum_{v\in \calv_2}r_vE^*_v$
is supported in $\calv\setminus \calv_2$, but $(x,E_v)=0$ for any $v\in \calv\setminus \calv_2$; hence $x=0$ since the intersection form of any subgraph is non--degenerate, see  Lemma \ref{lem:modform} too.

(b)
We construct the following graph $\calt_{1,u}'$ with arrowheads:
$\calt_{1,u}'$ consists of
all the vertices and edges of $\calt_{1,u}$, and we also add $\delta_{2,u}$
arrowheads attached to $u$ (that is, we replace
 the $u$--adjacent edges from $\calt_2$ by arrowheads).
Each arrowhead represent a cut (of $E_u$) in
  $\widetilde{X}(\calt_{1,u})$.
  We regard their collection as a multilink, that is, we endow the vertices and arrowheads with a multiplicity system (of  $\calt_{1,u}'$) as in \ref{ss:terminology}.
We define  the $m_v$--multiplicities as the
 multiplicities of $dE_u^*$ restricted to $\calv_{1,u}$, where  $d=\det(-(\,,\,)_\calt)$.
That is,  the  multiplicity of a vertex $v$ is
$dE^*_u|_v=dE^*_v|_u=-d(E^*_v,E^*_u)\in\Z_{>0}$.
Then the sum of the multiplicities of the arrowheads (all of them at $u$)
should be $c_u=d+d\sum_w
E^*_u|_w$, where the sum runs over the adjacent vertices $w$ of $u$ in $\calt_2$
(there are $\delta_{2,u}$ of them). Since each $ dE^*_u|_w\geq 1$, we get
$c_u\geq d+\delta_{2,u}$, hence
we can distribute $c_u$ into $\delta_{2,u}$ positive integers, such that
one of them is 1. These integers will be the multiplicities of the arrowheads
(link components).

The constructed
 multiplicity system defines a fibred multilink (cf. \cite{EN}).
Since one of the multiplicities is 1, the corresponding Milnor fiber is connected
(see also \cite[Th. 11.3]{EN}).
 Furthermore, the monodromy zeta function of the Milnor fibration is
(by A'Campo's theorem \cite{AC} or \cite{EN})
$\zeta(t)=\prod_{v\in\calv_{1,u}}(1-t^{dE^*_v|_u})^{\delta_v-2}$.
Hence, comparing the definitions of $\zeta$ and  $Z_{1,u}$,
 and using (\ref{eq:staridentity}), we get
that the reduced series of $Z_{1,u}$ is obtained by the following substitution:
 \begin{equation}\label{eq:zetaev}
 Z_{1,u}(\bt_{\calv_2})=\zeta(t)|_{t\mapsto \bt_{\calv_2}^{E^*_u|_{\calv_2}/ dE^*_u|_u}}.\end{equation}
We claim that if  $\delta_{2,u}\geq 2$ then  $\zeta(t)$ is a
polynomial. Indeed, being a zeta function of a connected Milnor fiber,
it has the form $\Delta(t)/(t-1)$, where $\Delta(t)$ is the
characteristic polynomial of the monodromy
of the first homology of the Milnor fiber. Hence, $\zeta$ is a polynomial if and only if it has no pole at $t=1$. But, since $\calt_{1,u}$ is a tree,  the vanishing order
${\rm ord}_{t-1}\zeta(t)=\sum_{v\in\calv_{1,u}}(\delta_v-2)=-2+\delta_{2,u}
\geq 0$.

Furthermore, the degree of $\zeta$ is
$\sum_{v\in\calv_{1,u}} (\delta_v-2)dE^*_v|_u$,
and  the (rational) $E^*_u|_{\calv_2}$ degree of $Z_{1,u}(\bt_{\calv_2})$
is $\deg(\zeta)/(dE^*_u|_u)$.
  Finally, by  (a) and its proof, all contribution in  the coefficient of $E^*_u|_{\calv_2}$
in $Z(\bt_{\calv_2})$ comes from $Z_{1,u}(\bt_{\calv_2})$.
\end{proof}

\subsection{}
Recall that the cycles $\{-E^*_v\}_{v\in\calv}$, considered as column vectors of a matrix, form
the inverse $(\,,\,)^{-1}$ of the intersection form. A similar property is valid
for the restrictions $\{-E^*_v|_{\calv_2}\}_{v\in\calv_2}$.

For a graph $\calt$ we say that a bilinear form $(\,,\,)_{mod}$ of $L\otimes \Q$
 is a {\it modified intersection form} of $(\,,\,)=(\,,\,)_\calt$,
 if $(E_v,E_w)_{mod}=(E_v,E_w)$ for any $v\not=w$ (and the diagonal might be modified,
 usually into some  rational entries).

\begin{lemma}\label{lem:modform}\ \cite[Lemma 11 (iii)]{LSz}
Let $\calv_2$ be as in \ref{ss:5.1}. The $|\calv_2|$--rank matrix $\{-E^*_v|_{\calv_2}\}_{v\in\calv_2}$ is the inverse of a negative definite matrix
$(\,,\,)_{mod}$, a modified intersection form of $(\,,\,)_{\calt_2}$.
(In fact, all the diagonal entries,  which are modified are indexed by   $\calb$.)
\end{lemma}
The proof is based on a diagonalization procedure of $(\,,\,)_\calt$ from \cite[\S 21]{EN}.

\subsection{} Let $\calt$ be as in \ref{ss:5.1}, and let us fix  a subset $\cali\subset \calv$, $\cali\not=\emptyset$.
The closure $\overline{\cali}$ of $\cali$ is defined as the
set of vertices of that connected minimal full subgraph of $\calt$ which contains $\cali$.

The following proposition  was first proved
(with slightly weaker bound) in \cite{LSz}
using residue formulae for vector partitions of \cite{SzV}. Here we provide an independent proof.

\begin{prop}\label{prop:LSz}
Assume that $l_0'\in \sum_v(\delta_v-2)E^*_v+\calS'$, $[l'_0]=h$.
Then \ $q^\calt_{h,\cali} (l'_0)=q^\calt_{h,\overline{\cali}} (l'_0)$.
\end{prop}

\begin{proof}
Let us write $\calt_2$ for the full connected subgraph with $\calv_2=\calv(\calt_2)=\overline{\cali}$, and we  adopt the notations of \ref{ss:5.1} associated with $\calt_2$.
Furthermore, we use the following
notations as well: $l'_0=\sum_va_vE^*_v$ is the fixed element of $L'$ appearing in the statement, $l'=\sum_vb_vE^*_v$ is an element from the support
${\rm Supp}(Z_h)$ of $Z_h$ (i.e. $z^\calt(l')\not=0$), and $l=l'-l'_0\in L$ with $l|_{\cali}\prec 0$.
(Such $l'$  parametrize the support of the sum in $q^\calt_{h,\cali} (l'_0)$.)
Write $c_v=b_v-a_v$ and $l=\sum_vm_vE_v$ (hence  $\{m_v,c_v\}_v$
is a multiplicity system in the sense of \ref{ss:terminology}).
The assumption $l|_{\cali}\prec 0$ reads as $m_v<0$ for all $v\in\cali$.

We wish to compare the sets
$\{l'\in {\rm Supp}(Z_h)\,:\,  (l'-l'_0)|_{\cali}
\prec 0\}$ and $\{l'\in {\rm Supp}(Z_h)\,:\,  (l'-l'_0)|_{\overline{\cali}}
\prec 0\}$ for fixed $l'_0$. If they agree then definitely we get
$q^\calt_{h,\cali} (l'_0)=q^\calt_{h,\overline{\cali}} (l'_0)$
(since we sum over the same set).
 The point is that these two sets can be  different, however, we show that
 the sum of the coefficients over the support--difference is zero.

To start the proof, let us fix some $l'\in {\rm Supp}(Z_h)$ with 
$(l'-l'_0)|_{\cali} \prec 0$.

First, we check an easy inequality.
Let us take $v\in\calv$ with $\delta_v>1$.
Then, by the assumption of \ref{prop:LSz},
$a_v\geq \delta_v-2$. But, by the shape of the rational function $f(\bt)$
from (\ref{eq:1.1}), $b_v\leq \delta_v-2$. Hence
\begin{equation}\label{eq:cv}
c_v\leq 0 \ \ \mbox{whenever}  \ \ \delta_v>1.\end{equation}
\noindent {\it The proof in the `easy case'.}   Assume that $m_u<0$ for all $u\in \calb$.
We claim that $m_v<0$ for all $v\in \overline {\cali}$, hence
$q^\calt_{h,\cali} (l'_0)=q^\calt_{h,\overline{\cali}} (l'_0)$ by the above discussion.

 Assume that this is not the case, and choose a maximal connected full subgraph $\calt'$
of $\calt_2$ with all $m_v$--multiplicities non--negative.

Let $C(l)=\sum_v(m_vE_v+c_vC_v)$ be the homologically trivial relative cycle in $\widetilde{X}$ associated with $l$ (with some choices of cuts $C_v$) as in
\ref{ss:terminology}.
Then $C(l) $ induces a relative cycle and a multiplicity system
 of $\calt'$  via the multiplicity preserving intersection
 $C(l)\cap \widetilde{X}(\calt')$, as it is explained in \ref{ss:terminology}.

By construction $\calv(\calt')\subset \overline{\cali}\setminus \cali$ and
for all $v\in \calv(\calt')$ one has $\delta_v>1$.
Therefore, those cut--multiplicities which come as restrictions of cuts
of $C(l)$  are  $\leq 0$ by (\ref{eq:cv}). The other cut--multiplicities, which come from the restriction of some neighboring $E_v$'s have multiplicities
$m_v<0$ (by the maximality of $\calt'$). Therefore,
 the restriction of $C(l)$ to $\calt'$ has all
 cut--multiplicities $\leq 0$, with at least
one $<0$ (at the `boundary' of $\calt'$). On the other hand,
 all $E_v$--multiplicities $\geq 0$. These facts contradict
the last sentence  of \ref{ss:terminology}.

\noindent
 {\it The proof in the  general case.}  Assume that $m_u\geq 0$ for at least
 one $u\in\calb$.

 Let us define the rational coefficients $\{c_{1,u}\}_{u\in \calv_2}$ as follows:
 \begin{equation}\label{eq:DEFcv}
 c_{1,u}:=\left\{\begin{array}{cc}
 \sum_{v\in \calv_{1,u}} c_v E^*_v|_u/ E^*_u|_u & \ \mbox{if} \ u\in \calb\\
 c_u & \ \,\mbox{if} \ u\not \in\calb.\end{array}\right. \end{equation}
Then,  by (\ref{eq:staridentity}), for each $u\in\calb$,
 $\sum_{v\in \calv_{1,u}} c_v E^*_v|_{\calv_2}= c_{1,u} E^*_u|_{\calv_2}$, hence
 \begin{equation}\label{eq:sumlI}
 \sum_{v\in \calv}c_vE^*_v|_{\calv_2}=\sum_{v\in\calv_2}c_{1,v}E^*_v|_{\calv_2}=
 \sum_{v\in\calv_2}m_vE_v|_{\calv_2}=l|_{\calv_2}.
 \end{equation}
\begin{clm}
There exists $u\in \calb$ with $m_u\geq 0$, $c_{1,u}>0$ and $\delta_{2,u}\geq 2$.
(For the definition of $\delta_{2,u}$ see Lemma \ref{lem:zetadeg}(b).)\end{clm}
 \begin{proof}
 Set  $\calv_2^{<0}:=\{v\in\calv_2\,:\, m_v< 0\}$
 and $\calv_2^{\geq 0}:=\{v\in\calv_2\,:\, m_v\geq 0\}$.
 By assumptions $\cali\subset \calv_2^{<0}$ and $\calv_2^{\geq 0}\not=\emptyset $ too.
 Write $l|_{\calv_2}$ as $l_1-l_2$, $l_1$ supported in $\calv_2^{\geq 0}$, while
  $l_2$ supported in $\calv_2^{< 0}$, both effective. Consider also the negative definite
  modified intersection form
  $(\,,\,)_{mod}$ associated with $\{-E^*_v|_{\calv_2}\}_{v\in \calv_2}$, defined in Lemma
  \ref{lem:modform}.  If $l_1\not=0$ then $(l|_{\calv_2},l_1)_{mod}\leq (l_1,l_1)_{mod}<0$,
  hence there exists $u\in \calv_2^{\geq 0}$ (in the support of $l_1$) such that $(l|_{\calv_2},E_u)_{mod}<0$.
  If $l_1=0$, since $\calt_2$ is connected, one can find $u\in \calv_2^{\geq 0}$ such that
  $E_u$ intersects the support of $l|_{\calv_2}$,
  hence $(l|_{\calv_2},E_u)_{mod}<0$ again.
  But, via (\ref{eq:sumlI}, $(l|_{\calv_2},E_u)_{mod}=-c_{1,u}$. Hence, there exists
  $u\in \calv_2^{\geq 0}$ such that $c_{1,u}>0$.
  Using (\ref{eq:cv}) and  definition (\ref{eq:DEFcv}) we get that $u\in\calb$ necessarily.
 On the other hand,  $\delta_{2,u}\geq 2$ too.
  Indeed, if  $u\in\calb$ and $\delta_{2,u}=1$
then  $u\in \cali$ (since
 $\calv_2=\overline{\cali}$ is  the closure of $\cali$) hence $m_u<0$.
\end{proof}

Let us introduce the coefficients  $\{b_{1,v}\}_{v\in\calv_2}$
associated with $\{b_v\}_{v\in\calv}$ by similar definitions as
(\ref{eq:DEFcv}).
The assumption regarding $a_v$'s, and  $c_{1,u}>0$, we obtain that
the $E^*_u|_{\calv_2}$--coefficient $b_{1,u}$ of $l'|_{\calv_2}$ satisfies
$$b_{1,u}= \sum_{v\in \calv_{1,u}} b_v E^*_v|_u/ E^*_u|_u >
\sum_{v\in \calv_{1,u}} a_v E^*_v|_u/ E^*_u|_u\geq d_{1,u}:=
\sum_{v\in \calv_{1,u}} (\delta_v-2) E^*_v|_u/ E^*_u|_u.$$
In particular, by Lemma \ref{lem:zetadeg},
$l'|_{\calv_2}$  is not  in the support of $Z_h(\bt_{\calv_2})$.

 This fact  can be reorganized as follows.
We order $\{u\in \calb\,:\, \delta_{2,u}\geq 2\} =\{u_1,\ldots ,u_s\}$.
We set $${\rm Supp}_1:=
\{l'\,:\, [l']=[l'_0],\
l'|_{\cali}\prec l'_0|_{\cali}, \ b_{1,u_1}>d_{1,u_1}\}$$
 and for $s\geq j>1$
 $${\rm Supp}_j:=
\{l'\,:\, [l']=[l'_0],\
l'|_{\cali}\prec l'_0|_{\cali}, \ b_{1,u_k}\leq d_{1,u_k} \ \mbox{ for $k<j$ },\ \ b_{1,u_j}>d_{1,u_j}\}.$$
Consider the restriction function
$\pi_j:{\rm Supp}_j\to L(\calt_2)\otimes \Q$, $l'\mapsto
l'|_{\calv_2}$. Then the sum
$\sum z^\calt(l') $ over any of the fiber of $\pi_j$ is zero.
Indeed, if we write $Z(\bt_{\calv_2})$ as $Z_2(\bt_{\calv_2})\cdot \prod_{u\in \calb}
Z_{1,u}(\bt_{\calv_2})$, as in the proof of Lemma \ref{lem:zetadeg},
then $Z_{1,u_j}(\bt_{\calv_2})$
 collects the contribution from
$\calt_{1,u_j}$ (as in the proof of \ref{lem:zetadeg}). Then in the fiber of $\pi_j$
the coefficient  $b_{1,u_j}$ is larger
than the  $E^*_{u_j}|_{\calv_2}$--degree of $Z_{1,u_j}(\bt_{\calv_2})$, and
by Lemma \ref{lem:zetadeg}(b)  $l'|_{\calv_2}$  is  not in the support
of $Z_h(\bt_{\calv_2})$. This means that the sum of the corresponding coefficients is zero.
In particular,
 the corresponding sum over all ${\rm Supp}_j$ is zero for any $j$.

Hence, up to these zero sums in the `modified counting function',
we can consider only the   cycles $l'$ from $ \{  {\rm Supp}(Z_h)\,:\,
(l'-l'_0)|_{\cali} \prec 0\}\setminus \cup_j {\rm Supp}_j $.
But by the above discussion such a cycle satisfies $m_u<0$ for all $u\in \calb$, hence
$m_u<0$ for all $u\in\overline{\cali}$ (by the `easy case').
Hence $q^\calt_{h,\cali} (l'_0)=q^\calt_{h,\overline{\cali}} (l'_0)$.
\end{proof}

\section{A surgery formula for modified counting functions}\label{s:SurgMCF}

\subsection{}\label{ss:setupMCF}
Choose some  $v\in\calv$, and
  let $\calt\setminus v=\cup_{k} \calt_{v,k}$ be the connected components of $\calt\setminus v$. Let  $j^*_{v,k}:L'(\calt)\to L'(\calt_{v,k})$
   be the dual operator defined similarly as $j^*_i$ above.
\begin{lem}\label{lem:EL} Fix one of the components, say $\calt_{v,k'}$, and
let $\calj\subset \calv(\calt_{v,k'})$, $\calj\not=\emptyset$. Then
for any $l_0'\in \sum_v(\delta_v-2)E^*_v+\calS'$ one has
\begin{equation}\label{eq:proof8}
q^\calt_{h,\calj}\,(l'_0)-q^\calt_{h,\calj\cup v}\,(l'_0)=
q^{\calt_{v,k'}}_{[j^*_{v,k'}(l'_0)],\calj}\,(j^*_{v,k'}(l'_0)).
\end{equation}
\end{lem}
\begin{proof} We will prove Lemma \ref{lem:EL} in three steps.

\bekezdes\label{ss:delta2}
First we assume that $v$ is an end--vertex of $\calt$, and the adjacent vertex $w$ has $\delta_w=2$. Partly,
we will follow the strategy of the proof of Proposition 3.2.4 of \cite{NJEMS}.
We will write $j^*$ for the dual operator, and we will
use the notations of the proof of Proposition  \ref{prop:LSz}:
$l'=\sum_ub_uE^*_u$, $l'_0=\sum_ua_uE^*_u$, $l=l'-l'_0\in L$, $l=\sum_uc_uE^*_u=\sum_um_uE_u$. Define also $Supp(\calt)=
\{l'\,:\, l|_{\calj}\prec 0,\ l|_v\geq 0\}\cap Supp(Z^\calt(\bt))$,
$Supp(\calt\setminus v)=
\{\tilde{l}'\,:\, \tilde{l}'|_{\calj}\prec j^*(l'_0)|_{\calj},\
[\tilde{l}']=[ j^*(l'_0)]\}\cap Supp(Z^{\calt\setminus v}
(\bt_{\calv\setminus v}))$.
In the left (resp. right) hand side of
(\ref{eq:proof8}) we sum over $Supp(\calt)$ (resp.
$Supp(\calt\setminus v)$).

In order to identify the coefficients of  $Z^\calt(\bt)$
and $Z^{\calt\setminus v}(\bt_{\calv\setminus v})$
easier it
is convenient to make the change of variables $x_u:=\bt^{E^*_u}$ ($u\in\calv$),
hence $Z^\calt(\bt)$ becomes $Z^\calt(\bx)=\prod_u(1-x_u )^{\delta_u-2}$.
In particular,
\begin{equation}\label{eq:substit}
Z^\calt(\bx)=Z_0(\bx_0)\cdot (1-x_v)^{-1}, \ \ \mbox{and} \ \ \
Z^{\calt\setminus v}(\bx)=Z_0(\bx_0)\cdot (1-x_w)^{-1},
\end{equation}
where $Z_0(\bx_0)$ is a series in variable $\{x_u\}_{u\in\calv\setminus \{v,w\}}$
only.

Since $\delta_w=2$, if $l'\in Supp(Z^\calt(\bt))$ then  $b_w=0$.

If we apply $j^*$ to the identity $l'-l'_0=l$ we get
$$j^*(l')-j^*(l'_0)=-m_vE^*_w+\sum_{u\in\calv\setminus v}m_uE_u. $$
Hence, $\Phi:Supp(\calt)\to Supp(\calt\setminus v)$, $l'\mapsto j^*(l')+m_vE^*_w$
is well--defined. Write
 $\bar{l}'=\sum_{u\not=w,v}b_uE^*_u$, $\bar{l}'_0=\sum_{u\not=w,v}a_uE^*_u$.
 Then   $l'=\bar{l}'+b_vE^*_v$ and
 $\Phi(l')=\bar{l}'+m_vE^*_w$.  Since $m_v=-(l,E^*_v)$ one has
 \begin{equation}\label{eq:mv}
 m_v=-(\bar{l}'-\bar{l}'_0,E^*_v)+a_w(E^*_w,E^*_v)-(b_v-a_v)(E^*_v,E^*_v).
 \end{equation}
 Hence $\Phi$ (that is, $(\bar{l}',b_v)\mapsto (\bar{l}',m_v)$) is injective.
Furthermore, for any $(\bar{l}',m_v)\in Supp(\calt\setminus v)$ the equation (\ref{eq:mv}) provides a unique well--defined candidate for $b_v$, such that $(\bar{l}',b_v)$ satisfies all the requirement of the elements of $Supp(\calt)$
except maybe $b_v\geq 0$ (cf. (\ref{eq:substit})).
Define $Supp(\calt\setminus v)^{\geq 0}$ as a subset of
$Supp(\calt\setminus v)$ consisting of those elements $(\bar{l}',m_v)$ for which
$b_v $ computed via (\ref{eq:mv}) is $\geq 0$; that is,
 $Supp(\calt\setminus v)^{\geq 0}={\rm im}(\Phi)$.
Then, using the bijection $\Phi$ onto its image and (\ref{eq:substit})
$$\sum _{l'\in Supp(\calt)}\ z^\calt
(l')\ =\
\sum _{\tilde{l}'\in Supp(\calt\setminus v)^{\geq 0}}
\ z^{\calt\setminus v}
(\tilde{l}').
$$
Set $Supp(\calt\setminus v)^{< 0}:=Supp(\calt\setminus v)\setminus
Supp(\calt\setminus v)^{\geq 0}$. In order to finish the proof of
(\ref{eq:proof8}), we need
\begin{equation}\label{eq:proof9}
\sum _{\tilde{l}'\in Supp(\calt\setminus v)^{<0}}
\ z^{\calt\setminus v}
(\tilde{l}')=0.
\end{equation}
But under the above correspondence and (\ref{eq:substit}), this equals
$\sum z^\calt(l')$, where the sum is over
$\{l'\,:\, l|_\calj\prec 0, \ l|_v\geq 0,\ [l']=[l'_0],\ b_v<0\}$.
The condition $b_v<0$ guarantees that these elements do not belong to the support of $Z^\calt$, hence  $\sum z^\calt(l')=0$.

\bekezdes \label{ss:delta3}
 Next, we assume that $v$ is an end--vertex of $\calt$, and the adjacent vertex $w$ has  $\delta_w\geq 3$.
Then we reduce this case to the previous case \ref{ss:delta2}: first we blow up  the edge connecting $v$ and $w$, then we apply  \ref{ss:delta2}, then we blow down the newly created vertex. We have to verify that the modified counting functions
are stable with respect to these operations.

When we blow up the edge $(v,w)$, then we create a new graph, denoted by
$\overline{\calt}$ with a newly created base element $\bar{E}_{new}\in L(\overline{\calt})$.
There is a natural projection $\rho: L(\overline{\calt})\to L(\calt)$, and
$\rho^*:L'(\calt)\to L'(\overline{\calt})$, which satisfy the
projection formula $(\rho^*(l'),\bar{l})=(l',\rho(\bar{l}))$. In particular,
$\rho^*(E^*_u)=\bar{E}^*_u$ (with natural notations).
Hence if we denote by  $\overline{\calj}\subset \calv(\overline{\calt})$
the same index set as  $\calj\subset \calv(\calt)$, then
\begin{equation}\label{eq:blowup}
z^\calt_{[l'_0]}(l')=z^{\overline{\calt}}_{[\rho^*(l'_0)]}(\rho^*(l')) \ \ \
\mbox{and} \ \ \
q^\calt_{[l'_0],\calj}(l'_0)=q^{\overline{\calt}}_{[\rho^*(l'_0)],\overline{\calj}}
(\rho^*(l'_0)).\end{equation}
Next, $\overline{\calt}\setminus v$ is the graph obtained from $\calt\setminus v$
by blowing up the vertex $w$. By definition, $Z^{\overline{\calt}\setminus v}$
has the shape $\rho^*(Z^{\calt\setminus v}(\bt))\cdot (1-\bt^{\bar{E}^*_w})/
(1-\bt^{\bar{E}^*_{new}})$, where $\rho^*(\sum z(l')\bt^{l'})=\sum z(l')\bt^{\rho^*(l')}$. Note that $\bar{E}^*_{new}=\bar{E}^*_w+\bar{E}_{new}$,
hence $\bt^{\bar{E}^*_{new}}=\bt^{\bar{E}^*_w} \cdot t_{new}$. Therefore, when we restrict to the $\overline{\calj}$ variables and  we substitute $t_{new}=1$,
the term $(1-\bt^{\bar{E}^*_w})/(1-\bt^{\bar{E}^*_{new}})$
becomes 1. Hence, the coefficients of the reduced series associated with
$\calt\setminus v$ and $\overline{\calt}\setminus v$ can be compared as in the previous case (\ref{eq:blowup}).

\bekezdes\label{ss:generalcase}
Now we consider the general situation. We prove the (\ref{eq:proof8}) by induction over $\sum_{k\not=k'}|\calv(\calt_{v,k})|$.
If this sum is zero, then we apply case \ref{ss:delta3}.
Consider the general situation, and let $e\in \cup_{k\not=k'}\calv(\calt_{v,k})$
be an end vertex of $\calt$.
Then, by cases \ref{ss:delta2}--\ref{ss:delta3}
\begin{equation}
q^\calt_{[l'_0],\calj}\,(l'_0)-q^\calt_{[l'_0],\calj\cup e}\,(l'_0)=
q^{\calt\setminus e}_{[j^*(l'_0)],\calj}\,(j^*(l'_0)),
\end{equation}
\begin{equation}
q^\calt_{[l'_0],\calj\cup v}\,(l'_0)-q^\calt_{[l'_0],\calj\cup v\cup  e}\,(l'_0)=
q^{\calt\setminus e}_{[j^*(l'_0)],\calj\cup v}\,(j^*(l'_0)).
\end{equation}
Since $q^\calt_{[l'_0],\calj\cup e}\,(l'_0)=q^\calt_{[l'_0],\calj\cup v\cup  e}\,(l'_0)$ by Proposition \ref{prop:LSz}, and
\begin{equation}
q^{\calt\setminus e}_{[j^*(l'_0)],\calj}\,(j^*(l'_0))=
q^{\calt\setminus e}_{[j^*(l'_0)],\calj\cup v}\,(j^*(l'_0))
+q^{\calt_{v,k'}}_{[j^*_{v,k'}(l'_0)],\calj}\,(j^*_{v,k'}(l'_0))
\end{equation}
by the inductive step, the identity (\ref{eq:proof8}) follows.
This ends the proof of Lemma \ref{lem:EL} as well.
\end{proof}
\section{The proof of Theorem \ref{prop:1}}\label{sec:MCF}

\subsection{}\label{ss:4.3}
 We will prove the identity (\ref{eq:proof1}) by induction on the cardinality
$|\cali|$ of $\cali$.

\bekezdes \label{ss:Iegy}
Assume that   $\cali$ contains exactly one element, say $v$.
We will use the notations of \ref{ss:setupMCF}, $\calt\setminus v=\cup_k \calt_{v,k}$. We have to prove
\begin{equation}\label{eq:proof11}
Q^\calt_{[l'_0]}\,(l'_0)-Q^\calt_{[l'_0],v}\,(l'_0)=\ \sum_k \
Q^{\calt_{v,k}}_{[j^*_{v,k}(l'_0)]}(j^*_{v,k}(l'_0)).
\end{equation}
We rewrite this identity in terms of modified counting functions (as in \ref{ss:4.4}).
For the last sum we have to consider
nonempty subsets $\calj\subset \calv\setminus v$. Hence, it is natural to organize the  nonempty subsets of $\calv $ as $\{v\}\cup \{\calj,\calj\cup v\}_{\calj\subset \calv\setminus v,\ \calj\not=\emptyset}$. The modified counting function associated with $\{v\}$  cancels with the second term
$Q^\calt_{[l'_0],v}\,(l'_0)$  of (\ref{eq:proof11}).
Hence the left hand side of (\ref{eq:proof11}) becomes an alternating sum of expressions of type
$q^\calt_{[l'_0],\calj}\,(l'_0)-q^\calt_{[l'_0],\calj\cup v}\,(l'_0)$.
If $\calj$ is not contained totally in only one $\calv(\calt_{v,k})$
then $v\in \overline {\calj}$, hence this expression  is zero
by Proposition \ref{prop:LSz}. Therefore we can assume that there exists
$k$ such that $\calj\subset \calv(\calt_{v,k})$. Hence
the expression (\ref{eq:proof11}) decomposes as a sum over $k$ according to this inclusion. Then the needed identity is the subject of   Lemma \ref{lem:EL}.

\bekezdes
Next, we take $\cali$ with  $|\cali|\geq 2$, and we assume, by the inductive step, that the identity
 $(\ref{eq:proof1})$ is true for any graph $\calt'$, any $h'\in H(\calt')$, and
  any subset $\cali'\subset
  \calv(\calt')$ with  $|\cali'|<|\cali|$.

Recall that $\calt\setminus \cali=\cup_i \calt_i$, and we wish to prove
\begin{equation}\label{eq:proof21}
Q^\calt_{[l'_0]}\,(l'_0)=Q^\calt_{[l'_0],\cali}\,(l'_0)+\sum_i \
Q^{\calt_i}_{[j^*_i(l'_0)]}(j^*_i(l'_0)).
\end{equation}
We choose some  $v\in\cali$, and we apply the inductive step for $\calt\setminus v$ and $\cali\setminus v$.
  Let $\calt\setminus v=\cup_{k} \calt_{v,k}$ be the connected components of $\calt\setminus v$, and let $j^*_{v,k}$ 
the corresponding dual operators. 
  Note that if $\calt_i$ is contained in $\calt_{v,k}$ then
   $j^*_{\calt_i\subset \calt_{v,k}}\circ j^*_{v,k}=j^*_i$.
 In particular, we get the following identity
 \begin{equation}\label{eq:proof22}
\sum_k\ Q^{\calt_{v,k}}_{[ j^*_{v,k}(l'_0)]}(j^*_{v,k}(l'_0))=
\sum_k\ Q^{\calt_{v,k}}_{[ j^*_{v,k}(l'_0)],\cali\setminus v}(j^*_{v,k}(l'_0))+
\sum_i \
Q^{\calt_i}_{[j^*_i(l'_0)]}(j^*_i(l'_0)).
\end{equation}
 By induction this identity is valid for any $\tilde{l}'_0$ (instead of
$j^*_{\calt\setminus v\subset \calt} (l'_0)
=\oplus_k j^*_{v,k} (l'_0)$) from the lattice
of $\calt\setminus v$ (satisfying the required assumptions that its $E^*$--coefficients are sufficiently high).
Hence it is true also for $\tilde{l}'_0=\oplus_k j^*_{v,k}(l'_0)$,
and this identity,  in this way,  will be
considered as a quasipolynomial identity in variable
$l'_0\in L'$ (cf. discussion from  \ref{ss:discussion}).
The difference between (\ref{eq:proof21}) and (\ref{eq:proof22}) is
\begin{equation}\label{eq:proof23}
Q^\calt_{[l'_0]}\,(l'_0)-\sum_k\ Q^{\calt_{v,k}}_{[ j^*_{v,k}(l'_0)]}(j^*_{v,k}(l'_0))=
Q^\calt_{[l'_0],\cali}\,(l'_0)-
\sum_k\ Q^{\calt_{v,k}}_{[ j^*_{v,k}(l'_0)],\cali\setminus v}(j^*_{v,k}(l'_0)).
\end{equation}
This identity (via induction) is equivalent with (\ref{eq:proof1}).
But, for the left hand side of (\ref{eq:proof23}) one can apply (the already proved)  (\ref{eq:proof11}).
%
In particular, (\ref{eq:proof1}) is equivalent with
 \begin{equation}\label{eq:proof5}
Q^\calt_{[l'_0],\cali}\,(l'_0)-Q^\calt_{[l'_0],v}(l'_0)=
\sum_k\ Q^{\calt_{v,k}}_{[ j^*_{v,k}(l'_0)],\cali\setminus v}(j^*_{v,k}(l'_0)).
\end{equation}
Next, we rewrite  the identity (\ref{eq:proof5})
 in terms of modified counting functions by the same principle as in
 \ref{ss:Iegy}.
For the last sum we have to consider
nonempty subsets $\calj\subset \cali\setminus v$, and  we
 organize the  nonempty subsets of $\cali$ as $\{v\}\cup \{\calj,\calj\cup v\}_{\calj\subset \cali\setminus v,\ \calj\not=\emptyset}$. The modified counting function associated with $\{v\}$  cancels with the second term of
 (\ref{eq:proof5}).
Hence the left hand side of (\ref{eq:proof5}) is
again a combination of expressions  of type
$q^\calt_{h,\calj}(l'_0)-q^\calt_{h,\calj\cup v}(l'_0)$.
If  $\calj$ is not contained totally in only one $\calv(\calt_{v,k})$
then $v\in \overline {\calj}$, hence this expression  is zero
by Proposition \ref{prop:LSz}. Hence we can assume that there exists
$k$ such that $\calj\subset \calv(\calt_{v,k})$ and
the expression (\ref{eq:proof5}) also decomposes as a sum over $k$ according to this inclusion, and it becomes  the statement of  Lemma \ref{lem:EL}.


\section{The proof of the vanishing from Example \ref{ex:partcases}(3)}\label{s:ExProof}

\subsection{Normal surface singularities} (For more details see \cite{trieste,NCL,Nfive,LPhd}).
Assume that $(X,o)$ is a complex analytic normal surface singularity, and let $\phi:\widetilde{X}\to X$ be a good resolution of $(X,o)$.
We denote the exceptional curve $\phi^{-1}(0)$ by $E$, and let $\cup_vE_v$ be
its irreducible components. Let $\calt$ be the dual resolution graph
associated with $\phi$ (which is automatically connected and negative definite).
Then $\widetilde{X}$, as a smooth manifold,
serves as the plumbing  4--manifold associated with $\calt$,
and $M=\partial \widetilde{X}$ is the plumbed
3--manifold (and also the `link' of $(X,o)$).
A resolution is minimal if there is no rational $E_v$ with $E_v^2=-1$.
We will assume, similarly as above, that  $M$ is a rational homology sphere, and we will use the notations from the previous sections.

The group of  analytic line bundles on $\widetilde{X}$
(up to isomorphism),
${\rm Pic}(\widetilde{X})$, appears in the exact sequence
\begin{equation}\label{eq:PIC}
0\to {\rm Pic}^0(\widetilde{X})\to {\rm Pic}(\widetilde{X})\stackrel{c_1}
{\longrightarrow} L'\to 0, \end{equation}
where  $c_1$ denotes the first Chern class of a line bundle. Furthermore,
$ {\rm Pic}^0(\widetilde{X})=H^1(\widetilde{X},\calO_{\widetilde{X}})\simeq
\C^{p_g}$, where $p_g$ is the {\it geometric genus} of
$(X,o)$. $(X,o)$ is called {\it rational} if $p_g(X,o)=0$.
 Artin in \cite{Artin62,Artin66} characterised rationality topologically
via their graphs. Such graphs are called `rational'.

The homomorphism
$c_1$ admits a unique (group homomorphism) section $l'\mapsto \calO(l')\in {\rm Pic}(\widetilde{X})$, such that $c_1(\calO(l'))=l'$, which extends the natural
section $l\mapsto \calO_{\widetilde{X}}(l)$ valid for integral cycles $l\in L$.

We say that for a  singularity $(X,o)$ and resolution $\phi$
the Seiberg--Witten Invariant Conjecture
(SWIC) is valid (cf. \cite{NN1,NPS})  if for any  $l'_0\in L'$ one has
\begin{equation}\label{eq:SWIC}
Q^{\calt} _{[l'_0]}(l'_0)+\mathfrak{sw}_{[-l_0']*\sigma_{can}}
(M(\calt))+
\frac{(K + 2l_0')^2+
 |\calv|}{8}= -h^1(\widetilde{X},\calO(-l_0')).
 \end{equation}
For rational singularities (and for any resolution of them) the SWIC  is valid,
cf. \cite{trieste,NCL,BN}.

The identity (\ref{eq:SWIC}) connects topological invariants of $(X,o)$ (left hand side) with analytic sheaf--cohomology invariants. There are two special regions
 for $l'_0$ when it simplifies. When $l'_0\in -K+\calS'$
  then by Generalized
 Grauert--Riemenschneider vanishing  theorem $h^1(\widetilde{X},\calO(-l'))=0$.
 Hence (\ref{eq:SWIC}) identifies the counting function with the
 normalized  Seiberg--Witten invariant (as in Theorem \ref{th:JEMS}).

However, when $\calS'\cap \{l'\,:\, l'\not\geq l'_0\}=\emptyset$, then
$ Q^{\calt} _{[l'_0]}(l'_0)=0$ and the rank of
 the corresponding sheaf--cohomology is  identified with the
 normalized  Seiberg--Witten invariant.
 If  both conditions are satisfied simultaneously
 then we obtain the vanishing of the normalized Seiberg--Witten invariants.

\bekezdes{\bf Cyclic quotient singularities \cite{BPV}.}\label{bek:CQS}
 Recall that $(X,o)$ is called a
cyclic quotient singularity if one of the following (equivalent) facts hold:

(1) $(X,o)$ is the quotient of $(\bC^2,0)$ by a cyclic group;

(2) the  graph $\calt$ of the minimal resolution is a string;

(3) there exists a finite map $p:(X,o)\to (\bC^2,0)$, whose (reduced)
discriminant (ramification locus)  is included in the union of the two local
coordinate axes of $(\bC^2,0)$.

Assume that $(X,o)$ is a cyclic quotient singularity, and $\phi$ is its minimal resolution. Let $C_1$ and $C_2$ be two cuts of the end--vertices.
Then there exists a finite projection $p:(X,o)\to (\bC^2,0)$ such that the discriminant of $p$ is included in
$\cup_{i=1,2} p(\phi(C_i))$, and  $ \{p(\phi(C_i))\}_{i=1,2}$ might serve as local coordinate axes of $(\bC^2,0)$.

Cyclic quotient singularities are rational.

\subsection{The proof of the vanishing from Example \ref{ex:partcases}(3).}
\label{s:PVan}
We start with a normal surface singularity $(X,o)$, a fixed resolution
 $\phi$ with dual graph $\calt$. We fix $h\in H$ and $r_h\in L'$ as above.
 We write $\calt\setminus \caln=\cup_i\calt_i$.
 Then each $\calt_i$ is a connected string, hence by contraction of the corresponding exceptional divisors indexed by $\calt_i$ we obtain a cyclic quotient singularity $(X_i,0)$. For this singularity
 the SWIC is valid. This,
 for the line bundle $\calO(-j^*_i(r_h))$,  reads as
\begin{equation*}
Q^{\calt_i} _{[j^*_i(r_h)]}(j^*_i(r_h))+\mathfrak{sw}_{[-j^*_i(r_h)]*\sigma_{can}}
(M(\calt_i))+
\frac{(K(\calt_i) + 2j^*_i(r_h))^2+
 |\calv(\calt_i)|}{8}= -h^1(\widetilde{X}(\calt_i),\calO(-j^*_i(r_h))).
 \end{equation*}
Write $\calv_i=\calv(\calt_i)$. Let $\{E_v\}_{v\in\calv_i}$ be the exceptional curves indexed  by $\calt_i$, and let $\partial\calv_i$ be those nodes
of $\calt$ which are adjacent with $\calv_i$ (this set
contains one or two elements).
If $n\in \partial \calv_i$ then let $w(n)\in\calv_i$
adjacent with $n$. Then, if $r_h=\sum_{v\in\calv} l'_vE_v$ then
$j^*_i(r_h)=\sum_{v\in \calv_i} l'_vE_v-\sum _{n\in \partial \calv_i} l'_n
E^*_{w(n)}$. Since each $l'_v\in[0,1)$, the cycle $j^*_i(r_h)$ has the form
$r_{h_i}-l$, where $r_{h_i}$ is in the
semi--open cube of $L'(\calt_i)$ and $l\in L(\calt_i)$, $l\geq 0$. In particular,
$\{l'\in L'(\calt_i)\,:\, l'\not\geq j^*_i(r_h)\}\cap \calS'(\calt_i)$ is empty
and $Q^{\calt_i}_{[j^*_i(r_h)]} (j^*_i(r_h))=0$.

Hence, the needed vanishing is equivalent with $h^1(\widetilde{X}(\calt_i),
\calO(-j^*_i(r_h)))=0$. Usually, by `standard'
vanishing theorems, see e.g. \cite[Th.12.1]{Lipman},
if in the resolution  of a rational singularity $l'\in \calS'$ then
$h^1 (\calO(-l'))=0$.
However, in this case $j^*_i(r_h)\in\calS'(\calt_i)$
is not necessarily true ($j^*_i(r_h)$ might have even negative
$E$--coefficients), see Example \ref{ex:graph1},
hence we need another (deeper) argument.

The proof relies on the structure of the
{\it universal abelian covering} (UAC) of
$(X,o)$. Since $H=H_1(M,\bZ)$ is finite, the abelianization $\pi_1(M)\to H$
determines a regular covering of $M_a\to M$, and
 a normal surface singularity $(X_a,o)$ with link $M_a$, and a finite analytic covering $c:(X_a,o)\to (X,o)$
with ramification locus only at $o\in X$. It is called the UAC of $(X,o)$ (see e.g. \cite{trieste,NWsq,OUAC} and the
references therein).

If $\phi:\widetilde{X}\to X$ is a good resolution of $(X,o)$ then let
$c':Z\to \widetilde{X}$ be the normalized pullback of
$c$ via $\phi$.
The (reduced) branch locus of $c'$ is included
in $\phi^{-1}(o)=E$, and the Galois action of $H$ extends to $Z$ as well.
Since $E$ is a normal
crossing divisor, the only singularities what $Z$ might have are
cyclic quotient singularities. Let
$\psi:\widetilde{X}_a\to Z$ be  a resolution of these singular points
such that $(c'\circ \psi)^{-1}(E)$ is a normal crossing divisor. Set $\widetilde{c}:=c'\circ \psi$.
\begin{equation}\label{eq:diagramUAC}
\begin{array}{ccccc}
\widetilde{X}_a & \stackrel{\psi}
{\longrightarrow} &Z& \longrightarrow & (X_a,o) \\
\Big\downarrow \vcenter{%
\rlap{$\scriptstyle{\widetilde{c}}$}}
 & &
\Big\downarrow \vcenter{%
\rlap{$\scriptstyle{c'}$}} & &
\Big\downarrow \vcenter{%
\rlap{$\scriptstyle{c}$}} \\
\widetilde{X} & = &\widetilde{X} & \stackrel{\phi}{\longrightarrow}  & (X,o)
\end{array} \end{equation}

The point is that the cycles $r_h\in L'$ and the line bundles $\calO(-r_h)\in
{\rm Pic}(\widetilde{X})$ appear in a natural way via the UAC as follows:
$\widetilde{c}_* \calO_{\widetilde{X}_a}$ has an $H$--eigenspace decomposition \cite{trieste,NPS,NCL,OUAC}
\begin{equation}\label{eq:RHDEC}
\widetilde{c}_* \calO_{\widetilde{X}_a}=\oplus_{h\in H} \, \calO(-r_h).
\end{equation}
Let $\widetilde{X}_i$ be a small neighbourhood of $\cup_{v\in\calv_i}E_v$ in $\widetilde{X}$.
It  serves as the plumbed 4--manifold $\widetilde{X}(\calt_i)$ associated with
$\calt_i$, and the restriction of $\phi$ is a minimal resolution
$\phi_i:\widetilde{X}_i\to X_i$ of the quotient singularity $(X_i,o)$.
Consider the restriction $\calO(-r_h)|_{\widetilde{X}_i}\in {\rm Pic}(\widetilde{X}_i)$. Its Chern class is $-j^*_i(r_h)\in L'(\calt_i)$. Since
$p_g(X_i,o)=0$, by the exact sequence (\ref{eq:PIC}) $\calO(-r_h)|_{\widetilde{X}_i}$ is determined by its Chern class, hence it is
$\calO(-j^*_i(r_h))$. Hence we need to prove that $h^1(\calO(-r_h)|_{\widetilde{X}_i})=0$.

Let $\widetilde{X}_{a,i}$ be $\widetilde{c}^{-1}(\widetilde{X}_i)$ in $\widetilde{X}_a$, and $\widetilde{c}_i:\widetilde{X}_{a,i}\to \widetilde{X}_i$
the restriction of $\widetilde{c}$.  Then the $H$ action preserves $\widetilde{X}_{a,i}$,
and the eigenspace decomposition (\ref{eq:RHDEC}) is compatible with the restriction, hence $\calO(-r_h)|_{\widetilde{X}_i}$ is a direct
eigenspace summand (corresponding to $h$) of $(\widetilde{c}_{i})_*\calO_{\widetilde{X}_{a,i}}$.
Hence it is enough to prove that
$h^1((\widetilde{c}_i)_*\calO_{\widetilde{X}_{a,i}})=0$.
 Since $\psi$ resolves only cyclic quotient singularities, and $c'$ is finite,
 $R^1\widetilde{c}_*\calO_{\widetilde{X}_a}=0$. Hence, by Leray spectral sequence,
 $h^1((\widetilde{c}_i)_*\calO_{\widetilde{X}_{a,i}})=
 h^1(\calO_{\widetilde{X}_{a,i}})$. Thus, we need
$h^1(\calO_{\widetilde{X}_{a,i}})=0$.

$\widetilde{X}_{a,i}$ has several (isomorphic) connected components.
 By construction, $\widetilde{c}_i$ is a regular covering off
$(\cup_{v\in\calv_i}E_v)\cup\cup_{n\in\partial \calv_i}(E_n\cap \widetilde{X}_i)$.
The disc(s) $C_n:=E_n\cap \widetilde{X}_i$ are/is cut(s) of $\calt_i$ in $\widetilde{X}_i$ at the end--vertices. By the discussion from \ref{bek:CQS}
there is a projection $p_i:(X_i,0)\to (\bC^2,0)$, such that
$p_i(\cup_n C_n)$ is included in the discriminant, which itself  is  included
in the union of the coordinate axes. This $p_i$ composed with $\widetilde{c}_i$
provides a map $\widetilde{X}_{a,i}\to (\bC^2,0)$ with discriminant included in
the union of coordinate axes.
 Hence,  each component of $\widetilde{X}_{a,i}$ is a
resolution of a cyclic quotient singularity. Since cyclic quotient singularities are rational, $h^1(\calO_{\widetilde{X}_{a,i}})=0$.

\begin{remark}\label{rem:pg}
From (\ref{eq:RHDEC}) one has  $p_g(X_a,o)=\sum_{h\in H} h^1(\widetilde{X},
\calO(-r_h))$.
\end{remark}

\begin{example}\label{ex:graph1}
Let $\calt$ be the left graph below,  and at right we show
the $E$--multiplicities of  $r_h\in L'$.

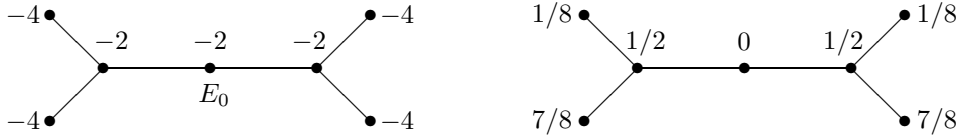
\begin{picture}(300,55)(-20,-15)
\put(20,10){\circle*{4}} \put(60,10){\circle*{4}}
\put(100,10){\circle*{4}}
\put(20,10){\line(1,0){80}}
\put(20,10){\line(-1,1){20}} \put(20,10){\line(-1,-1){20}}
\put(100,10){\line(1,1){20}}\put(100,10){\line(1,-1){20}}
\put(0,30){\circle*{4}} \put(0,-10){\circle*{4}}
\put(120,30){\circle*{4}} \put(120,-10){\circle*{4}}
\put(23,20){\makebox(0,0){$-2$}}
\put(60,20){\makebox(0,0){$-2$}}
\put(97,20){\makebox(0,0){$-2$}}
\put(-10,30){\makebox(0,0){$-4$}}
\put(-10,-10){\makebox(0,0){$-4$}}
\put(130,30){\makebox(0,0){$-4$}}
\put(130,-10){\makebox(0,0){$-4$}}

\put(220,10){\circle*{4}} \put(260,10){\circle*{4}}
\put(300,10){\circle*{4}}
\put(220,10){\line(1,0){80}}
\put(220,10){\line(-1,1){20}} \put(220,10){\line(-1,-1){20}}
\put(300,10){\line(1,1){20}}\put(300,10){\line(1,-1){20}}
\put(200,30){\circle*{4}} \put(200,-10){\circle*{4}}
\put(320,30){\circle*{4}} \put(320,-10){\circle*{4}}
\put(223,20){\makebox(0,0){$1/2$}}
\put(260,20){\makebox(0,0){$0$}}
\put(297,20){\makebox(0,0){$1/2$}}
\put(188,30){\makebox(0,0){$1/8$}}
\put(188,-10){\makebox(0,0){$7/8$}}
\put(332,30){\makebox(0,0){$1/8$}}
\put(332,-10){\makebox(0,0){$7/8$}}
\put(60,0){\makebox(0,0){ $E_0$}}
\end{picture}

Let $\calt_i$ be the subgraph consisting of the
$(-2)$ vertex $E_0$ between the two nodes. Then $j^*_i(r_h)=
-E_0^*=-E_0/2$. Hence, usually $j^*_i(r_h)$ is not even effective.
 The Chern class of $\calO(-j^*_i(r_h))$  is $(E_0/2,E_0)=-1$.
  Since $h^1(\calO(-E_0+E_0/2))=0$ by
 Grauert--Riemenschneider type vanishing,
 we get $h^1(\calO(E_0/2))=h^1(\calO_{E_0}(E_0/2))=
 h^1(\calO_{\P^1}(-1))=0$.
\end{example}

\begin{example}\label{ex:graph2}
Consider the following graph  $\calt$  and
the $E$--multiplicities of certain $r_h\in L'$.

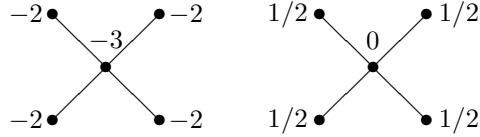
\begin{picture}(300,55)(-100,-15)
\put(20,10){\circle*{4}}
\put(20,10){\line(-1,1){20}} \put(20,10){\line(-1,-1){20}}
\put(20,10){\line(1,1){20}}\put(20,10){\line(1,-1){20}}
\put(0,30){\circle*{4}} \put(0,-10){\circle*{4}}
\put(40,30){\circle*{4}} \put(40,-10){\circle*{4}}
\put(20,20){\makebox(0,0){$-3$}}
\put(-10,30){\makebox(0,0){$-2$}}
\put(-10,-10){\makebox(0,0){$-2$}}
\put(50,30){\makebox(0,0){$-2$}}
\put(50,-10){\makebox(0,0){$-2$}}

\put(120,10){\circle*{4}}
\put(120,10){\line(-1,1){20}} \put(120,10){\line(-1,-1){20}}
\put(120,10){\line(1,1){20}}\put(120,10){\line(1,-1){20}}
\put(100,30){\circle*{4}} \put(100,-10){\circle*{4}}
\put(140,30){\circle*{4}} \put(140,-10){\circle*{4}}
\put(120,20){\makebox(0,0){$0$}}
\put(88,30){\makebox(0,0){$1/2$}}
\put(88,-10){\makebox(0,0){$1/2$}}
\put(152,30){\makebox(0,0){$1/2$}}
\put(152,-10){\makebox(0,0){$1/2$}}
\end{picture}

Let $\cali$ be the union of the four $(-2)$--vertices, hence $\calt_1= \calt\setminus \cali$ consists of the $(-3)$--vertex $E_0$.
Then $j^*_1(r_h)= -2E_0^*=-2E_0/3$.
Its representative $r_{h_1}$ is
$E_0/3$,  hence, usually, $j^*_i(r_h)\not=r_{h_i}$.

 The Chern class of $\calO(-j^*_1(r_h))$  is $(2E_0/3,E_0)=-2$.
 Hence, 
   $h^1(\calO(2E_0/3))=h^1(\calO_{E_0}(2E_0/3))=
 h^1(\calO_{\P^1}(-2))=1$.
 Furthermore,
 $Q^{\calt_i} _{[j^*_1(r_h)]}(j^*_1(r_h))=0$, thus
 \begin{equation}\label{eq:100}
\mathfrak{sw}_{[-j^*_1(r_h)]*\sigma_{can}}
(M(\calt_1))+
\frac{(K(\calt_1) + 2j^*_1(r_h))^2+
 |\calv(\calt_1)|}{8}= -1.
 \end{equation}
 %
\end{example}

\section{Application:  $\calt_i$ are rational} \label{s:8}

\subsection{}\label{ss:8.1}
Consider  the situation of Theorem \ref{th:MainResult} and assume that all
$\calt_i$ are rational (see \cite{Nfive} or  Section
\ref{s:ExProof}). We will prove two `reduction formulae', see Propositions
\ref{prop:red1} and \ref{prop:red2}.

If $h=0$ (hence $r_h=0$ too)  then
the SWIC is valid for the corresponding singularity,  hence (\ref{eq:SWIC}) applied for $l'_0=0$ reads as
\begin{equation}\label{eq:SWICrat}
\mathfrak{sw}_{\sigma_{can}}(M(\calt_i))+
\frac{K(\calt_i)^2+  |\calv(\calt_i)|}{8}= 0.
 \end{equation}
Hence, if each $\calt_i$ is rational then
\begin{equation*}
\mathfrak{sw}_{\sigma_{can}}(M)+\frac{K^2+|\mathcal{V}|}{8}
 = - \textnormal{pc}^{\pi_\cali(\calS'_{\R})}(Z_0(\mathbf{t}_\cali)),
\end{equation*}
thus the `normalized Seiberg--Witten invariant associated with $\sigma_{can}$
can be computed as the periodic constant of the series reduced to the variables
$\bt_\cali$.

In general, for arbitrary $h$,
 the  vanishing  (\ref{eq:vanishing}) does not hold (even if
$\calt$ itself  is rational), cf. Example \ref{ex:graph2}.
However, the contribution from $\calt_i$ rational still can be simplified.

In order to state the results we need some preparation.

\subsection{}\label{ss:8.2}
 We fix a graph as in \ref{ss:PGP} and $h\in H$.
Then there exists a unique representative $s_h\in \calS'\subset L'$ of $h$,  which is the unique minimal element  (with respect to the partial ordering)
of $\{s\in\calS'\,:\, [s]=h\}$  \cite{NOSZ,trieste}.
Usually $s_h\not=r_h$. Since $s_h\geq 0$, by the definition of $r_h$
 we have $s_h-r_h=\Delta_h\in L$,
$\Delta_h\geq 0$.

Note that for $s_h$ still $Q^\calt_{h}(s_h)=0$ (since $Z$ is supported on $\calS'$, and in $\calS'$ the representative
$s_h$ is minimal in its class). Moreover, if $\calt$ is rational, then for $s_h$ applies
 Lipman's vanishing  as well \cite[Th.12.1]{Lipman}, namely
$h^1(\widetilde{X},\calO(-l'))=0$   for any $l'\in \calS'$.
Thus, the SWIC for $l'_0=s_h$ reads as
\begin{equation}\label{eq:SWIC2}
\mathfrak{sw}_{-h*\sigma_{can}}
(M(\calt))+
\frac{(K + 2s_h)^2+
 |\calv|}{8}= 0 \ \ (\mbox {whenever $\calt$ is rational}).
 \end{equation}

\begin{remark}\label{rem:lc}
Since the lattice cohomology theory is a categorification of the normalized Seiberg--Witten invariants cf. \cite{NJEMS}, the above facts can also be reinterpreted by the lattice cohomological characterization of rational
singularities (for more details see \cite[4.1]{Nlat}).
\end{remark}

\subsection{}\label{ss:8.3}
Consider the situation of Theorem \ref{th:MainResult} with $\calt_i$ rational.
Then (\ref{eq:SWIC2}) and \ref{rem:harom}(2) imply
\begin{equation}\label{eq:vanishing2}
 \mathfrak{sw}_{-[j^*_{i}(r_h)]*\sigma_{can,i}}(M_i)+\frac{(K(\mathcal{T}_i) + 2j^*_i(r_h))^2+
 |\mathcal{V}(\calt_i)|}{8}=\chi(s_{h_i})-\chi(j^*_{i}(r_h)).
\end{equation}
Hence, if all $\calt_i$ are rational then
\begin{proposition}\label{prop:red1}
 \begin{equation*}
\mathfrak{sw}_{-h*\sigma_{can}}(M)+\frac{(K+2r_h)^2+|\mathcal{V}|}{8}
 =\
 - \textnormal{pc}^{\pi_\cali(\calS'_{\R})}(Z_{h}^\calt(\mathbf{t}_\cali))
 + \ \sum_{i} \Big( \chi(s_{h_i})-\chi(j^*_{i}(r_h)) \Big).
\end{equation*}
\end{proposition}
\subsection{} \label{ss:8.4}
Consider again
the situation of Theorem \ref{th:MainResult} with all $\calt_i$ rational.
Corollary \ref{cor:QP} applied for $l'_0=s_h=r_h+\Delta_h$ reads as
\begin{equation}\label{eq:last}
\begin{split}
\ & \ \mathfrak{sw}_{-h*\sigma_{can}}(M)+\frac{(K+2s_h)^2+|\mathcal{V}|}{8}
 = \\
\ & \sum_{i} \Big( \mathfrak{sw}_{-[j^*_{i}(s_h)]*\sigma_{can,i}}(M_i)+\frac{(K(\mathcal{T}_i) + 2j^*_i(s_h))^2+
 |\mathcal{V}(\calt_i)|}{8}\Big)
 - \mathfrak{Q}^\calt_{h,\cali}\,(\Delta_h).
 \end{split}
\end{equation}
The following fact follows directly from definitions, for details see
 \cite[Lemma 8.4.2]{LSz}.

\begin{lemma}
$j^{*}_i(s_{h}) = s_{[j^{*}_i(s_{h})]}$ in $L'_{\mathcal{T}_i}$.
\end{lemma}
%
In particular, (\ref{eq:SWIC2}) applied for each $\calt_i$ gives the vanishing
of the $\sum_i$ in (\ref{eq:last}). Moreover, by Proposition \ref{prop:PCcut} one has
$$\mathfrak{Q}^\calt_{h,\cali}\,(\Delta_h)=
(\bt^{-\Delta_h|_\cali}Z_h^\calt(\bt_\cali)|_{\not\geq 0})({\bf 1})+
{\rm pc}^{\pi_\cali(S'_\R)}(\bt^{-\Delta_h|_\cali}Z_h^\calt(\bt_\cali)|_{\geq 0}).
$$
Hence we obtain the following reduction formula
\begin{proposition}\label{prop:red2}
\begin{equation*}
 \mathfrak{sw}_{-h*\sigma_{can}}(M)+\frac{(K+2s_h)^2+|\mathcal{V}|}{8} =
  - (\bt^{-\Delta_h|_\cali}Z_h^\calt(\bt_\cali)|_{\not\geq 0})({\bf 1})-
{\rm pc}^{\pi_\cali(S'_\R)}(\bt^{-\Delta_h|_\cali}Z_h^\calt(\bt_\cali)|_{\geq 0}).
\end{equation*}
(By \cite[(4.3.15)]{LN}
on the right hand side  one can replace $\Delta_h$ by $s_h$, in this way the series will be `genuine' series with integral exponents.)
\end{proposition}
\subsection{More examples and applications.} \label{ss:8.5}
(1)
The surgery formulae of this section generalize those surgery formulae,
which reduce the lattice $L$ to a lower rank lattice associated with
`bad vertices'. We recall that a collection of vertices
  $\cali$ of $\calv$
is called `bad' if by decreasing the decoration of these vertices
 on the graph  we obtain a
rational graph  (cf. \cite{NOSZ,LNRed}). Since the subgraph of a rational graph is rational, if $\cali$ consists of `bad vertices' then all components of $\calt\setminus \cali$ are rational. (Nevertheless, the converse is not true,
see e.g. examples from \cite[8.2(5)]{NOSZ}.) In this sense our new
surgery formula from Proposition \ref{prop:red2} generalizes
 \cite[Th. 5.3]{LNRed}.

(2)
A special family of graph manifolds when $\calt\setminus \cali$ are all rational  is provided by $S^3_{-d}(K)$, the $(-d)$-surgery along the connected sum $K=K_1\#\dots \#K_{\nu}\subset S^3$ of algebraic knots $K_{\ell}$. In this case there is a special vertex $v_+$ such that all the connected components of $\calt\setminus v_+$ represent  $S^3$. In this case $Z_0(t_{v_+})$ can be computed from the Alexander polynomials of the knots $K_{\ell}$, providing
explicit formula for the Seiberg--Witten invariants in terms of these Alexander polynomials.
For  details see  \cite[8.1]{BN} or \cite[Th. 2.4.5]{NR}.

\section{The case of numerically Gorenstein graphs}

\subsection{} Recall that Corollary \ref{cor:QP} assures that the
counting function $Q^\calt_{h,\cali}(l_0') $ and its quasipolynomial
$\mathfrak{Q}^\calt_{h,\cali}(l'_0-r_h)$ agree whenever all
$a_v$ coefficients of $l_0'$ are sufficiently large. In general, for an arbitrary
 graph and $h$ it is hard to determine a precise (and sharp) bound from which this
 equality holds. However, for numerically Gorenstein graphs and $h=0$
 we determine such a bound. The presentation also shows the perfect parallelism of our `topological dualities' with the (algebraic/analytic) Gorenstein (or Serre)
 dualities known in singularity theory or algebraic geometry.

 \subsection{Definitions and notations.}  The connected negative definite graph $\calt$ is called {\it numerically Gorenstein} if $K\in L$.
 In this section we  assume  that the graph is {\it minimal good} (that is, there exists no vertex $v$ with $E_v^2=-1$ and $\delta_v\leq 2$) and  numerically Gorenstein, and we consider $h=0$ only.
 We denote the anticanonical cycle $-K$  by $Z_K$.
 Note that  $Z_K=0$ if and only if
 $\calt$ is  ADE--graph (all decorations are $-2$), and in all other cases
 all the coefficients of $Z_K$ are strict positive \cite[Prop. 2.1]{Laufer},
 \cite[Cor. 2.8]{PPP}.

 Recall that $\chi:L\to \Z$ was defined as $-(l,l-Z_K)/2$, hence the first trace of the duality/symmetry is $\chi(l)=\chi(Z_K-l)$.

Motivated by the theory of lattice cohomology (see e.g. \cite{Nlat,NJEMS})
 we consider for any $\calj\subset \calv$
 lattice cubes $ (l,\calj)$, of dimension $|\calj|$, of the cubical decomposition given by
 $L\simeq \Z^{|\calv|}\subset \R^{|\calv|}$. The vertices of such a cube $(l,\calj)$ are
 $\{l+E_{\calj'}\}_{\calj'\subset \calj}$ (where
 $E_{\calj}=\sum_{v\in\calj}E_v$).
 The weight of the cube  $ (l,\calj)$  is defined as
 $w(l, \calj) = \max_{\calj' \subset \calj}\{\chi(l + E_{\calj'})\}$.

For lattice points $a \leq b$, $a,b \in L$, we define the rectangle $R(a, b)$ by
 $\{x\in L\otimes \R \,:\, a \leq x \leq b\}$. Then the cube
 $(l,\calj)$ belongs to $R(a,b)$ if $a\leq l+E_{\calj'}\leq b$ for
 all its vertices.

The cycle $j_\cali\pi_\cali Z_K\in L$ has the same $E_v$--coefficient as $Z_K$
whenever $v\in \cali$, otherwise it is zero. 
They define the rectangles $R(\cali):=R(j_\cali\pi_\cali Z_K, Z_K)$. E.g.,
$R(\emptyset) =R(0,Z_K)$ and $R(\calv)=R(Z_K,Z_K)$.

 In the next discussion it is convenient to use the next abridged notation
 for any connected $\calt$:
 $$\overline{\mathfrak{sw}}(\calt) = -\mathfrak{sw}_{\sigma_{can}}(M(\calt))- (K^2+|\mathcal{V}|)/8.$$
If $\calt$ has several connected components, say
 $\calt = \cup_{i}\calt_i$, then we set
  $\overline{\mathfrak{sw}}(\calt) = \sum_{i }\overline{\mathfrak{sw}}(\calt_i)$.

 The setup is as in Section \ref{s:3}:
 $\cali\subset \calv$ is non--empty,  and $\calt\setminus\cali=\cup_i\calt_i$.
  We write
 $Z^\calt_0(\bt)=\sum_{l\in L}z(l)\bt^l$.

 We start with the following immediate consequence of (\ref{eq:QPrestr})
 (use $j^*_i(K)=K(\calt_i)\in L'(\calt_i)$):
 \begin{proposition}\label{prop:duality}
 {\bf (Topological duality of the quasipolynomial)} The quasipolynomial
 $\mathfrak{Q}^\calt_{0,\cali}$ satisfies the symmetry
  $\mathfrak{Q}^\calt_{0,\cali}(l)=\mathfrak{Q}^\calt_{0,\cali}(Z_K-l)$,
  in particular
  ${\rm pc}^{\pi_\cali(\calS'_\R)}(Z_0(\bt_\cali))=
  \mathfrak{Q}^\calt_{0,\cali}(0)=\mathfrak{Q}^\calt_{0,\cali}(Z_K)$.
 \end{proposition}

\subsection{} In the  next formulae the rectangle  $R(0,Z_K)$ will play a crucial role:
basically we will express all our invariants as sums  of weighted cubes of different faces of $R(0,Z_K)$.

The main result of this section is the following.

\begin{thm}\label{th:Z_Kbound} Under the above assumptions
 $\mathfrak{Q}^\calt_{0,\cali}(Z_K)=Q^\calt_{0,\cali}(Z_K)$.
 In particular,
\begin{equation}\label{eq:Z_KBound}
\textnormal{pc}^{\pi_\cali(\calS'_{\R})}(Z_{0}(\mathbf{t}_\cali)) =
\mathfrak{Q}^\calt_{0 ,\cali}(Z_K) =Q^\calt_{0,\cali}(Z_K).
\end{equation}
\end{thm}

The main advantage of (\ref{eq:Z_KBound}) is that the
the needed correction term in the surgery formulae, the
usually hardly computable and more theoretical
$\textnormal{pc}^{\pi_\cali(\calS'_{\R})}(Z_{0}(\mathbf{t}_\cali)) $,
can be replaced by the  directly computable
$\sum _{l|_{\cali}\not\geq Z_K|_{\cali}}z(l)$. This shows that for such graphs the
quasipolynomial  $\mathfrak{Q}^\calt_{0 ,\cali}$ can be avoided.
\begin{proof}
First we recall some needed results.
\begin{fact}\ \cite[Th. 2.3.10]{NJEMS} For any $l\in L$
\begin{equation}\label{eq:zw}
z(l) = \sum_{\calj \subset \calv} (-1)^{|\calj|+1} w(l, \calj).
\end{equation}
\end{fact}

\begin{fact}\ \cite{NJEMS}, \cite[\S 5.3]{LNRed}
For any  $b \in L$ which satisfies  $b \geq Z_K$ one has
\begin{equation}\label{eq:Qb}
\overline{\mathfrak{sw}}(\calt) =
\sum_{(l, \calj) \subset R(0, b)} (-1)^{|\calj|+1} w(l, \calj).
\end{equation}
Furthermore, there is a combinatorial cancelation (`contraction')  of cubes,
 which identifies \begin{equation}\label{eq:Qbb}
Q^\calt_{0,\calv}(Z_K)=\sum_{l\not\geq Z_K} z(l)=\sum _{l\not\geq Z_K}\, \sum _{\calj\subset \calv}
(-1)^{|\calj|+1} w(l,\calj) \  \ \mbox{with} \ \
\sum_{(l,\calj)\subset R(\emptyset), \atop  l\not= Z_K} (-1)^{|\calj|+1} w(l,\calj).
\end{equation}
In particular, the two identities combined provide
\begin{equation}\label{eq:Qbbb}
Q^\calt_{0,\calv}(Z_K)=w(Z_K,\emptyset)+\overline{\mathfrak{sw}}(\calt)=
\chi(Z_K)+\overline{\mathfrak{sw}}(\calt)
=\overline{\mathfrak{sw}}(\calt).
\end{equation}
\end{fact}

This result was stated  for  connected graphs,
however it extends naturally to non--connected graphs as well by the
additivity of $ \chi$ and $Z_K$ over the connected components.

Note that  (\ref{eq:Qbbb}) together with (\ref{eq:SUMQP2})
and Proposition  \ref{prop:duality} imply Theorem \ref{th:Z_Kbound} for $\cali=\calv$, that is: $
\mathfrak{Q}^\calt_{0,\calv}(Z_K)= \mathfrak{Q}^\calt_{0,\calv}(0)=
\overline{\mathfrak{sw}}(\calt)=
Q^\calt_{0,\calv}(Z_K)$. 

The very same combinatorial cancelation of  (\ref{eq:Qbb})
from \cite{NJEMS}, \cite[\S 5.3, Lemma 5.3.3]{LNRed}
(with completely identical proof) provides the following identity as well.
\begin{fact}
For any $\cali\subset \calv$, $\cali\not=\emptyset$,  one has
  \begin{equation}\label{eq:QbbI}
q^\calt_{0,\cali}(Z_K) =
\sum_{(l,\calj)\subset R(\emptyset)\setminus \cup_{v\in\cali} R(\{v\})} (-1)^{|\calj|+1} w(l,\calj).
\end{equation}
(In the sum those cubes do not appear  which sit in the affine hyperplanes
$l|_v=Z_K|_v$ for some $v\in \cali$.)
\end{fact}

\subsection{}
Let us  introduce the function $\mathfrak{s}: \{\mbox{set of graphs}\} \to \Z$,
such that
\begin{equation}\label{eq:defh}
\sum_{\mathcal {T}' \subset \mathcal {T}}\mathfrak{s}(\mathcal {T}') =  \overline{\mathfrak{sw}}(\calt)
\ \ \mbox{(with the convention
$\mathfrak{s}(\emptyset)=\overline{\mathfrak{sw}}(\emptyset)=0$)}.
\end{equation}
By induction on $|\mathcal {V}(\calt)|$ one shows that $\mathfrak{s}$ is uniquely defined by
(\ref{eq:defh}). Moreover, the property
  $\overline{\mathfrak{sw}}(\cup_{i=1}^r\calt_i) =
  \sum_{i=1}^r\overline{\mathfrak{sw}}(\calt_i)$,
  valid for several connected components,  transforms into
$\mathfrak{s}(\cup_{i=1}^r\calt_i)=0$ whenever $r\geq 2$.
Furthermore, by combinatorial cancellation (or by M\"obius invertion)
\begin{equation}\label{eq:mobius}
\mathfrak{s}(\mathcal {T}) = \sum_{\cali \subset \calv} \,(-1)^{|\calv|- |\cali|} \, \overline{\mathfrak{sw}}(\calt(\cali)).
\end{equation}

\begin{lemma}\label{lem:hsum} For any $\cali\subset \calv$
the following identity holds.
\begin{equation}\label{eq:hosszu}
 \mathfrak{s}(\calt(\cali)) =
\sum_{(l,\calj)\subset R(\calv\setminus \cali)\setminus
\cup_{v\in\cali} R(\calv\setminus (\cali\setminus v))}
(-1)^{|\calj|+1} w(l, \calj).
\end{equation}
\end{lemma}
\begin{proof}

 Since for any $l_\cali\in L(\calt(\cali))$ one has $\chi_{\calt(\cali)}(l_\cali)=
 \chi(j_\cali(l_\cali))=\chi(Z_K-j_\cali(l_\cali))$, (\ref{eq:Qb}) implies
 \begin{equation}\label{eq:Qb2}
\sum_{(l, \calj) \subset R(\calv\setminus \cali)} (-1)^{|\calj|+1} w(l, \calj) =
\sum_{(l, \calj) \subset R(0,Z_K|_\cali)} (-1)^{|\calj|+1} w(l, \calj),
\end{equation}
where the second sum is considered in the lattice $L(\calt(\cali))$.
We claim that in this lattice  $D:=Z_K|_\cali- Z_K(\calt(\cali))\geq 0$.
Indeed, by the two adjunction formulae,
for any $v\in\cali$, $(E_v, D)=-(E_v, Z_K|_{\calv\setminus \cali})\leq 0$
(since $Z_K|_{\calv\setminus \cali}$ is effective),
hence $D\in\calS'(\calt(\cali))$, and $D\geq 0$. In particular, the
right hand side of (\ref{eq:Qb2}) via (\ref{eq:Qb}) is $\overline{\mathfrak{sw}}(\calt(\cali))$. This gives for any $\cali\subset \calv$
 \begin{equation}\label{eq:Qb3}
\overline{\mathfrak{sw}}(\calt(\cali))=
\sum_{(l, \calj) \subset R(\calv\setminus \cali)} (-1)^{|\calj|+1} w(l, \calj).
\end{equation}

 By  (\ref{eq:mobius}) and (\ref{eq:Qb3}) and combinatorial cancelation:
\begin{equation*}
\mathfrak{s}(\calt(\cali))=  \sum_{\calj\subset \cali} (-1)^{|\cali|-|\calj|}\cdot
\overline{\mathfrak{sw}}(\calt(\calj))=
 \sum_{\calj\subset \cali} (-1)^{|\cali|-|\calj|}
 \, \sum_{(l,\calk)\subset R(\calv\setminus \calj)}
 (-1)^{|\calk|+1}w(l,\calk),
\end{equation*}
which equals the right hand side of (\ref{eq:hosszu}).
\end{proof}
\noindent
Then, for any $\calj\not=\emptyset$,  (\ref{eq:QbbI}) and  Lemma \ref{lem:hsum} imply
$$\sum_{\calk\supset \calj} \mathfrak{s}(\calt(\calk))=q^\calt_{0,\calj}(Z_K).$$
Therefore, for any $\cali$,
\begin{equation*}
\begin{split}
Q^\calt_{0,\cali}(Z_K)&=\sum_{\emptyset\not=\calj\subset \cali }
(-1)^{|\calj|+1}q^\calt_{0,\calj}(Z_K)=\sum_{\emptyset\not=\calj\subset \cali }
(-1)^{|\calj|+1}\ \sum_{\calk\supset \calj} \mathfrak{s}(\calt(\calk))\\ &
=\sum_{\calk\subset\calv}\mathfrak{s}(\calt(\calk))\sum_{\emptyset\not=
\calj\subset \calk\cap\cali} (-1)^{|\calj|+1}
=\sum_{\calk\subset \calv,\, \calk\cap\cali\not=\emptyset} \mathfrak{s}(\calt(\calk)) \\ &
=\sum_{\calk\subset \calv} \mathfrak{s}(\calt(\calk))-
\sum_{\calk\subset \calv\setminus \cali} \mathfrak{s}(\calt(\calk))
=\overline{\mathfrak{sw}}(\calt)-\overline{\mathfrak{sw}}(\calt\setminus \cali)=
\textnormal{pc}^{\pi_\cali(\calS'_{\R})}(Z_{0}(\mathbf{t}_\cali)).
\end{split}
\end{equation*}
This ends the proof of Theorem \ref{th:Z_Kbound}.
\end{proof}
The formulae (\ref{eq:Qb}), (\ref{eq:Qbb}), (\ref{eq:QbbI}) and (\ref{eq:Qb3})
provide explicit expressions for
$\overline{\mathfrak{sw}}(\calt)$, $Q^\calt_{0,\calv}(Z_K)$,
$q^\calt_{0,\cali}(Z_K)$ and $\overline{\mathfrak{sw}}(\calt(\cali))$
in terms of weighted cubes of different faces of $R(0,Z_K)$.

\end{document}